\documentclass[a4paper,11pt,english]{smfart}
\usepackage{amsfonts}
\usepackage{amsmath,amsthm} 
\usepackage{latexsym}
\usepackage{array}
\usepackage{amssymb}
\usepackage{graphics}
\usepackage{enumerate}
\usepackage[english]{babel}

\usepackage{color}
\definecolor{marin}{rgb}   {0.,   0.3,   0.7} 
\definecolor{rouge}{rgb}   {0.8,   0.,   0.} 
\definecolor{sepia}{rgb}   {0.8,   0.5,   0.} 
\usepackage[colorlinks,citecolor=marin,linkcolor=rouge,
            bookmarksopen,
            bookmarksnumbered
           ]{hyperref}

\theoremstyle{plain} 
\newtheorem{theorem}{Theorem}[section]
\newtheorem{lemma}[theorem]{Lemma}
\newtheorem{proposition}[theorem]{Proposition}
\newtheorem{definition}[theorem]{Definition} \theoremstyle{remark}
\newtheorem{remark}[theorem]{Remark}

\newcommand{\R}{  \mathbb{R}   }

\newcommand{\eps}{\varepsilon}

\newcommand{\C}{  \mathbb{C}   }
\newcommand{\Z}{  \mathbb{Z}   }
\newcommand{\N}{  \mathbb{N}   }

\newcommand{\Hc}{  \mathcal{H}   }
\renewcommand{\O}{  \mathcal{O}   }
\renewcommand{\P}{  \mathcal{P}   }

\newcommand{\T}{  \mathbb{T}   }

\newcommand{\Sc}{  \mathcal{S}   }
\newcommand{\Uc}{  \mathcal{U}  }

\newcommand{\dd}{  \text{d}   }

\newcommand{\ov}{  \overline  }

\newcommand{\la}{  \lambda  }
\renewcommand{\l}{  \ell  }

\renewcommand{\phi}{  \varphi  }
\renewcommand{\L}{  \mathcal{L}   }

\newcommand{\dist}{\operatorname{dist}}

\newcommand{\ad}{\operatorname{ad}}

\newcommand{\be}{\begin{equation}}
\newcommand{\ee}{\end{equation}}
\newcommand{\ben}{\begin{equation*}}
\newcommand{\een}{\end{equation*}}

\newcommand{\Norm}[2]{\|#1\|\left.\vphantom{T_{j_0}^0}\!\!\right._{#2}}         
\newcommand{\SNorm}[2]{|#1|\left.\vphantom{T_{j_0}^0}\!\!\right._{#2}}     

\numberwithin{equation}{section}

\def\norma#1{\left\| #1\right\|}

\author{Dario Bambusi }
\address{Dipartimento di Matematica, Universit\`a degli studi di Milano\\
Via Saldini 50, I-20133 Milano}
\email{dario.bambusi@unimi.it}

\author{Erwan Faou}
\address{INRIA \& ENS Cachan Bretagne  \\
Avenue Robert Schumann F-35170 Bruz, France. } 
\email{Erwan.Faou@inria.fr}

 \author{ Beno\^it Gr\'ebert}
\address{Laboratoire de Math\'ematiques Jean Leray\\
2, rue de la Houssini\`ere \\
44322 Nantes Cedex 03, France}
\email{benoit.grebert@univ-nantes.fr}

\title[Numerical Solitons for DNLS]
{Existence and stability of solitons for fully discrete approximations of the nonlinear Schr\"odinger equation.}

\begin{document}

\begin{abstract}
In this paper we study the long time behavior of a discrete approximation in time and space of the cubic nonlinear Schr\"odinger equation on the real line. More precisely, we
consider a symplectic time splitting integrator applied to a discrete nonlinear Schr\"odinger equation with additional Dirichlet
boundary conditions on a large interval. 
We give conditions ensuring the existence of a numerical soliton which is close in energy norm to the continuous soliton. Such result is valid under a  CFL condition of the form $\tau h^{-2}\leq C$ where $\tau$ and $h$ denote
the time and space step size respectively. Furthermore we prove that if the  initial
datum is symmetric and close  to the continuous soliton $\eta$ then the associated numerical
solution remains close to the orbit of $\eta$,
$\Gamma=\cup_\alpha\{e^{i\alpha}\eta\}$, for very long
times.   
\end{abstract}

\thanks{
}

\maketitle

\section{Introduction}
We study numerical approximations of solitons of the focusing nonlinear Schr\"odinger equation (NLS) on the real line:
\be\label{nls}
i\psi_t=-\psi_{xx}-|\psi|^2\psi, \quad x\in\R,\ t\in\R.
\ee
This equation is a Hamiltonian partial differential equation (PDE) associated with the Hamiltonian function 
\begin{equation}
\label{hc}
H(\psi):=\int_\R \left[\left|\psi_x\right|^2-\frac{\left|\psi\right|^4}{2} \right]
\dd x, 
\end{equation}
and preserving the $L^2$ norm
\begin{equation}
\label{nc}
N(\psi):=\int_{\R} \left|\psi\right|^2 \dd x. 
\end{equation}
The goal of this paper is to understand the long time behavior of numerical
integration algorithms for initial data close to the solitary wave solution
$\psi(t,x)=e^{i\la t}\eta(x)$ where 
\begin{equation}
\label{phis}
\eta(x):=\frac{1}{\sqrt
2}{\rm sech}\left( \frac{x}{2}\right), 
\end{equation} 
and $\lambda \in \R$ is the Lagrange multiplier associated with the
minimization of $H$ under the constraint $N=1$. It is well known, see
for instance \cite{Weinstein85,Grill87,Grill90,Frohlich04} that this
solution is {\em orbitaly stable} in the sense that for a small
pertubation of the initial data, the exact solution remains close to
the orbit of $\eta$ for all times. Here we will only consider {\em
symmetric} initial conditions satisfying $\psi(x) = \psi(-x)$, a
property that is preserved by the flow of \eqref{nls}.  In this
setting, the orbital stability of the continuous soliton can be
described as follows: Let
\begin{equation}
\label{eq:defgamma}
\Gamma := \bigcup_{\alpha \in \R} \{Êe^{i\alpha} \eta(x)\}Ê
\end{equation}
and assume that $\psi(0,\cdot)$ is a symmetric function satisfying $\dist(\psi(0,\cdot),\Gamma) \leq \delta$ for some $\delta$ sufficiently small, then for all times $t > 0$, if $\psi(t,\cdot)$ denotes the solution of \eqref{nls}, we have 
\begin{equation}
\label{eq:bite}
\forall\, t > 0, \quad 
\dist(\psi(t,\cdot), \Gamma) < C \delta, 
\end{equation}
where $C$ is a constant independent of $\delta$ and $t$, and where the
distance is measure in $H^1$ norm.  The present paper deals with the
persistence of this result by fully discrete numerical methods. It is
an old problem that was pointed out in several papers in the last 30
years, see for instance \cite{DFP81,SZ86,Duran00,Borgna08}, and
the numerical approximation of \eqref{phis} over long times has now
become a classical benchmark to test the performance and stability of
numerical schemes, see for instance \cite{Akrivis,Fei95,Besse} and the
references therein. However, as far as we know, no result of the
form \eqref{eq:bite} has been proven in the literature for fully
discrete approximations of \eqref{nls} (see
however \cite{BP10,Borgna08} for the space discretized case).

In particular, the effect of the time discretization yields many mathematical difficulties. Dur\'an \& Sanz-Serna gave in \cite{Duran00} some asymptotic expansion of the numerical solution close to a soliton, but the lack of a modified energy acting on $H^1$ and preserved over long time by the numerical scheme (the so called {\em backward error analysis}) was an obstruction to define a possibly stable numerical soliton. Here, we take advantage of a recent construction of such a modified energy given by Faou \& Gr\'ebert in \cite{FG11} to show the existence and stability of a modified soliton that is close to \eqref{phis} in energy norm. 

In this paper, the discretization of \eqref{nls} we consider are made of three levels of approximations: 

\begin{itemize}
\item\textbf{A space discretization}, where we use a grid with mesh
size $h > 0$ made of an infinite collection of equidistant points of
$\R$. The equation \eqref{nls} is then approximated by the discrete
nonlinear Schr\"odinger equation (DNLS) where the Laplace operator is
replaced by its finite difference approximation over three points.  

\item \textbf{A Dirichlet cut-off}, where we replace the integrability condition at infinity of the derivative of $\psi$ by a Dirichlet boundary condition at the boundary of a large window of size $2Kh$ where $K >> 1$. 
\item \textbf{A time discretization algorithm} to integrate the  DNLS equation with Dirichlet boundary condition. This  discretization introduces a last parameter $\tau$ which represents the time step.  To do this we consider 
a symplectic time splitting algorithm where the kinetic part and
potential part are solved alternatively as described for instance
in \cite{Weideman86}. 
\end{itemize}
Each of these three levels of discretization relies on discretization parameters. In this paper, we prove orbital stability in the sense of \eqref{eq:bite} for the numerical solution, where the distance to $\Gamma$ is estimated in terms of the three discretization parameters $h$, $K$ and $\tau$. 

We first present some numerical experiments showing that the solitary
wave rapidly disappears if either the algorithm of integration is not symplectic, or if it is symplectic, but used with  a too large CFL number  $\tau h^{-2}$.

The proof is organized as follows: we first recall in Section 4 the main arguments of the proof of the orbital stability result in the continuous case, following in essence the presentation made in \cite{Frohlich04}. We then give in Section 5 an abstract result showing that if the energy space $H^1$ is well approximated by the space discretization, and if the numerical scheme preserves - or almost preserves - modified $L^2$ norm and energy functions that are close to the exact ones, we can obtain orbital stability results with precise bounds depending on the parameters. We then apply this formalism in Section 6 to our three levels of discretization. 

As the proof of orbital stability result is based on the variational characterization of the solitary wave and thus heavily relies on the preservation of the energy and $L^2$ norm, long time bounds can be straightforwardly obtained for energy and $L^2$ norm preserving schemes such as the Dufour-Fortin-Payre scheme, see \cite{DFP81}. This follows directly from the analysis of the space discretized case (see also Remark \ref{BFG}).

The cornerstone of the analysis of splitting method is the construction of the modified energy. Recall that in the finite dimensional case, the existence of modified energy is guaranteed by Hamiltonian interpolation: see \cite{BG94,HLW,Reich99} but cannot be applied straightforwardly to Hamiltonian PDEs unless unreasonable a priori assumptions are made on the regularity of the numerical solution, which prevents a fair use of the bootstrap argument underlying the orbital stability methodology. 
Here we take advantage of the recent backward error analysis result of \cite{FG11} to construct a modified energy {\em acting on $H^1$} for splitting methods applied to \eqref{nls}. Actually we give a simplified proof of a simpler
version of the result presented in \cite{FG11} or \cite{F11}, which has some interest in itself.

Using this result, we then prove an orbital stability result for fully discrete splitting method applied to \eqref{nls} with a CFL restriction, and over very long times of the form  $n\tau\sim \tau^{-M}$, where $M \geq 0$ is an integer number depending on the CFL. 

\section{Three discretization levels and main results}\label{theo}

We now describe more precisely the three levels of approximation of \eqref{nls} mentioned in the introduction. At each step, we state the orbital result that we obtain. 

\subsection{Space discretization}Ê
Having fixed a positive parameter $h$ we discretize space by
substituting the sequence
$\psi_{\l}\simeq\psi(h \l)$, $\l\in\Z$ for the function $\psi(x)$,  and the second order operator of finite difference
$\Delta_h$ defined by
\begin{equation}
\label{deltamu}
(\Delta_h \psi)_{\l}:=\frac{\psi_{\l+1}+\psi_{\l-1}-2\psi_{\l}}{h^2}, 
\end{equation}
for the Laplace operator $-\partial_{xx}$. 
The NLS is thus reduced to the discrete nonlinear Schr\"odinger
equation (DNLS):
\begin{equation}
\label{dnls} i\dot
\psi_\l =-\frac{1}{h^2}(\psi_{\l+1}+\psi_{\l-1}-2\psi_\l)-
|\psi_\l|^2\psi_\l , \quad \ell \in\Z \ .
\end{equation}
where $t \mapsto \psi(t) = (\psi_\ell(t))_{\ell \in \Z}$ is an application from $\R$ to $\C^\Z$.  With this equation is associated a Hamiltonian function and a discrete $L^2$ norm given by 
\begin{equation}
\label{eq:discrham}
H_h(\psi)=h\sum_{j\in \Z}\left[\left|\frac{\psi_j
-\psi_{j-1}}{h}\right|^2
-\frac{|\psi_j|^4} {2}\right]
\quad \mbox{and}\quad
N_h(\psi)=h\sum_{j\in \Z} |\psi_j|^2. 
\end{equation}
The discrete space of functions is 
$$
V_h = \{ \psi_j \in \C^{\Z}\, | \, \psi_j = \psi_{-j}\}
$$
equipped with the discrete norm 
$$
\Norm{\psi}{h} = 2h \sum_{j\in\Z} \frac{|\psi_{j+1}-\psi_j |^2}{h^2}+h \sum_{j\in\Z} |\psi_j|^2. 
$$
Following \cite{BP10}, we identify $V_h$ with a finite element subspace of
$H^1(\R;\C)$. More precisely,  defining the function $s:\R \to \R$ by 
\begin{equation}
\label{fe.1}
s(x)=
\begin{cases}
0\qquad\qquad  &{\rm if}\quad |x|>1,\\
x+1\quad  &{\rm if}\quad   -1\leq x\leq 0,\\
- x+1\quad  &{\rm if}\quad  0\leq x\leq 1,
\end{cases}
\end{equation}
 the identification is done through the map $i_h: V_h \to H^1(\R;\C)$ defined by 
\begin{equation}
\label{eq:ih}
\left\{\psi_j\right\}_{j\in\Z}\mapsto (i_h\psi)(x) := \sum_{j\in\Z} \psi_j \left(\frac{x}{h}-j\right) \ .
\end{equation}
Recall that $\Gamma$ is the curve of minima of the continuous Hamiltonian and is given by \eqref{eq:defgamma}.
With these notations, we have the following result
\begin{theorem}
\label{th:dnls}
There exist $\delta_0$ and $h_0$ such that  for all $\delta < \delta_0$ and $h \leq h_0$, if $(\psi^0)_{j \in \Z} \in V_h$ is such that
$$
\dist(i_h \psi^0,\Gamma) \leq \delta,
$$
where the distance is measured in the continuous $H^1(\R;\C)$ norm, then the solution $(\psi_j(t))_{j \in \Z}$ of \eqref{dnls}  satisfies
$$
\forall\, t \geq 0, \quad 
\dist(i_h \psi(t),\Gamma) \leq C (\delta + h)
$$
for some constant $C$ independent of $h$ and $\delta$. 
\end{theorem}
Notice that the DNLS flow is not defined globally everywhere, i.e. for all initial data in $V_h$ and all times $t$. However since a solution of DNLS issued from an initial datum close to $\Gamma$ remains unconditionally close to $\Gamma$,   such solution is automatically global.

\subsection{Dirichlet cut-off}
In order to come down to a finite dimensional system we fix a large number
$K \geq 1$, substitute the sequence $-K,...,K$ for the set $\Z$ in \eqref{dnls}, 
and add Dirichlet boundary conditions $\psi_{-K-1}=\psi_{K+1}=0$. The equation we consider is thus the (large) ordinary differential system 
\be 
\label{dnlsdir}
\left\{\begin{array}{rcl}
i\dot \psi_\l &=&-\displaystyle\frac{1}{\mu^2}(\psi_{\l+1}+\psi_{\l-1}-2\psi_\l)- |\psi_\l|^2\psi_\l ,\quad  -K\leq \l\leq K\\[2ex]
\psi_{\pm (K+1)} &=& 0. 
\end{array}\right.
\ee
Note that here, we use the convention that $\psi_\ell = 0$ for all
$|\ell| \geq K+1$, so that the previous system is indeed a closed set of
differential equations.  The corresponding discrete function space is
\begin{equation}
\label{eq:VhK}
V_{h,K} := \{ (\psi_j)_{j \in \Z} \in V_h \, | \, \psi_j =
0\quad \mbox{for} \quad |j| \geq K+1\},
\end{equation}
on which we can define the Hamiltonian function and discrete $L^2$ norm $H_{h,K} := H_h|_{V_{h,K}}$ and $N_{h,K}:= N_{h}|_{V_{h,K}}$ as restrictions of the functions \eqref{eq:discrham} to $V_{h,K} \subset V_h$. Similarly, we define $i_{h,K} = i_h|_{V_{h,K}}$. In the following, we often use the notation $(\psi_j)_{j =-K}^K$ to denote an element of $V_{h,K}$ with the implicit extension by $0$ for $|j| \geq K+1$ to define an element of \eqref{eq:VhK}. 
With these notations, we have the following result: 

\begin{theorem}
\label{th:dnlsdir}
There exist constants $C_1$, $C_2$, $\delta_0$ and $\epsilon_0$ such that for all $\delta < \delta_0$ and all $h$ and $K$ such that $ h + \frac{1}{h^2} e^{-C_1 Kh} \leq \epsilon_0$,  if $(\psi^0_j)_{j= -K}^K \in V_{h,K}$ is such that 
$$
\dist(i_{h,K} \psi^0,\Gamma) \leq \delta,
$$
then the solution $(\psi_j(t))_{j = -K}^K$ of \eqref{dnlsdir} satisfies
$$
\forall\, t \geq 0, \quad 
\dist(i_{h,K} \psi(t),\Gamma) \leq C_2 \Big(\delta + h + \frac{1}{h^2} e^{-C_1 Kh}\Big). 
$$
\end{theorem}

\begin{remark}
The exponentially small term in the previous estimate represents the effect of the Dirichlet cut-off. As we will see below, it  directly comes from the fact that the function $\eta$ is exponentially decreasing at infinity. 
\end{remark}

\subsection{Time discretization}
In this work the time discretization of \eqref{nls} that we consider is a splitting scheme: we construct $\psi^{n}$ the approximation of the solution $\psi(t)$ of \eqref{nls}  at time $n\tau$ iteratively by the formula  
$$
\psi^{n+1} =\Phi_A^\tau \circ \Phi_P^\tau(\psi^{n} ), 
$$ 
where the flow $\Phi_P^\tau$ is by definition the exact solution of 
$$
i \dot \psi_\ell = - |\psi_\ell|^2 \psi_\ell, \quad \ell = -K,\ldots,K \, , 
$$
in $V_{h,K}$ 
which is given explicitly by formula $\Phi_P^\tau( \psi)_\ell = \exp( i \tau |\psi_\ell|^2) \psi_\ell$. The flow $\Phi_A^\tau$, is by definition the solution of 
\begin{equation}
\label{eq:phiA}
i \dot \psi_\ell = -\frac{1}{h^2}(\psi_{\l+1}+\psi_{\l-1}-2\psi_\l), \quad \ell = -K,\ldots,KÊ\, , 
\end{equation}
with the convention $\psi_{\ell} = 0$ for $|\ell | \geq K+1$. The
implementation of this numerical scheme requires the computation of an
exponential of a tridiagonal matrix at each step. It could also be
done in discrete Fourier space in which the operator on right-hand side is diagonal. The
main advantage of this splitting method is that it is an explicit and
symplectic scheme.

Our main result is the following
\begin{theorem}
\label{th:splitting}
There exist constants $C_1$, $C_2$, $\delta_0$ and $\epsilon_0$ such that for all $\delta < \delta_0$ and all $h$, $\tau$ and $K$ such that $ h +  \frac{1}{h^2} e^{-C_1 Kh} \leq \epsilon_0$  and the following CFL condition is satisfied
\begin{equation}
\label{eq:CFL}
 (2 M + 3) \frac{\tau}{h^2} < \frac{2\pi}{3},
\end{equation}
then 
if $(\psi^0_j)_{j= -K}^K \in V_{h,K}$ is such that 
$$
\dist(i_{h,K} \psi^0,\Gamma) \leq \delta, 
$$
we have 
\begin{equation}
\label{eq:estSPLIT}
\forall\, n\tau  \leq \tau^{-M}, \quad 
\dist(i_{h,K} (\Phi^\tau_A \circ \Phi_P^\tau)^n \psi^0,\Gamma) \leq C_2 \Big(\delta + h + \frac{\tau}{h} + \frac{1}{h^2} e^{-C_1 Kh}\Big). 
\end{equation}
\end{theorem}

\begin{remark}
In the last estimate \eqref{eq:estSPLIT}, the term $\tau/h$ represents the error induced by the modified energy constructed with the method of \cite{FG11} (see Section 7 below). Note that under the condition \eqref{eq:CFL}, this term is actually of order $\mathcal{O}(h)$. 
\end{remark}
\begin{remark}\label{BFG}
An alternative  time approximation of \eqref{dnlsdir} is the modified Crank-Nicolson scheme given by Delfour-Fortin-Payre see \cite{DFP81,SZ84} defined as the application $\psi^n \mapsto \psi^{n+1}$ such that 
$$
\psi^{n+1}_\ell = \psi^{n}_\ell + \frac{i\tau}{2}(\Delta_h (\psi^{n+1} + \psi^{n}))_\ell + \frac{i \tau}{4} (|\psi^{n+1}_\ell|^2 + |\psi^{n+1}_\ell|^2) (\psi^{n+1}_\ell + \psi^{n}_\ell), 
$$
for $\ell = -K,\ldots,K$. 
It can be shown using a fixed point argument that for $\tau$ sufficiently small, $\psi^{n+1}$ is well defined, and that this scheme preserves exactly the discrete $L^2$ norm and discrete energy \eqref{eq:discrham}. Using this property, it can easily be shown that the conclusions of Theorem \ref{th:dnlsdir} extends straightforwardly to this specific fully discrete case.  Notice that this method  has the disadvantage to be strongly implicit.   
\end{remark}

\section{Numerical experiments}\label{num} 

In this section, we would like to illustrate the results given in
 Theorem \ref{th:splitting}, and prove that if the CFL
 condition \eqref{eq:CFL} is not satisfied, the stability
 estimate \eqref{eq:estSPLIT} is no longer true. In contrast, we
 show that if the CFL number is small enough, a numerical stability
 can be indeed observed.  On the other hand, we show that for non
 symplectic integrators, even used with a very small CFL number,
 numerical instabilities appear.

In a first example, we take $h = 0.1875$, $K = 80$ (so that $Kh = 15$),  $\tau = 0.2$ and the initial condition \eqref{phis}. The CFL number is equal to 5.7. We consider the integrator $\Phi_A^\tau \circ \Phi_P^\tau$ defined above. As mentioned in the previous section, the flow of $\Phi_P\tau$ can be calculated explicitely, while the computation of $\Phi_A$ - see \eqref{eq:phiA} - is performed using the  {\tt expm} MATLAB procedure. 

In Figure \ref{fig4}, we plot the absolute value of the fully discrete numerical solution $\psi^{n} = (\Phi_A^\tau \circ \Phi_P^\tau)^n (\psi^0)$. We can observe that the shape of the soliton is destroyed between the times $t = 100$ and $200$. 

\begin{figure}[ht]
\begin{center}
   \rotatebox{0}{\resizebox{!}{0.33\linewidth}{%
   \includegraphics{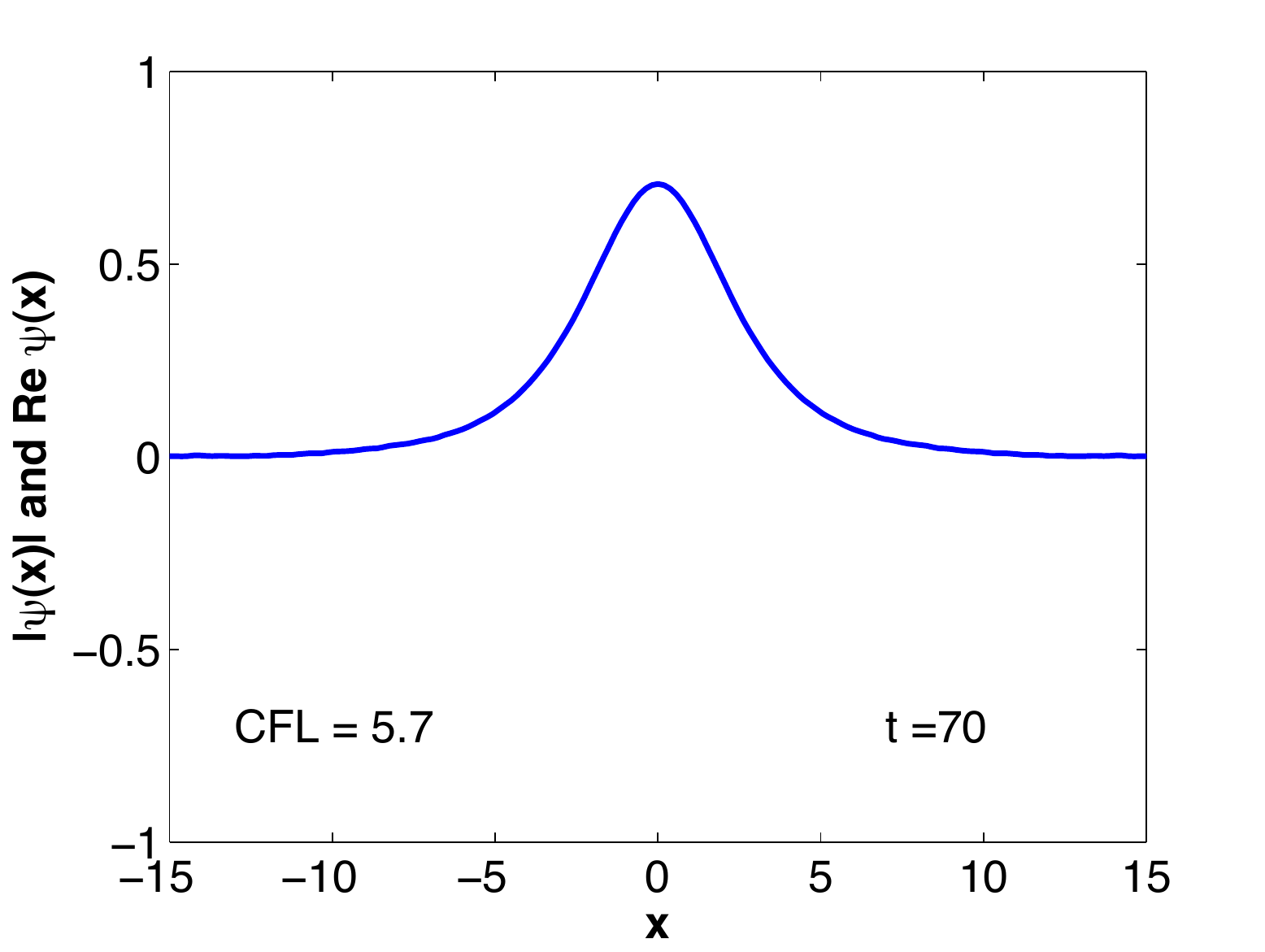}}}
\rotatebox{0}{\resizebox{!}{0.33\linewidth}{%
   \includegraphics{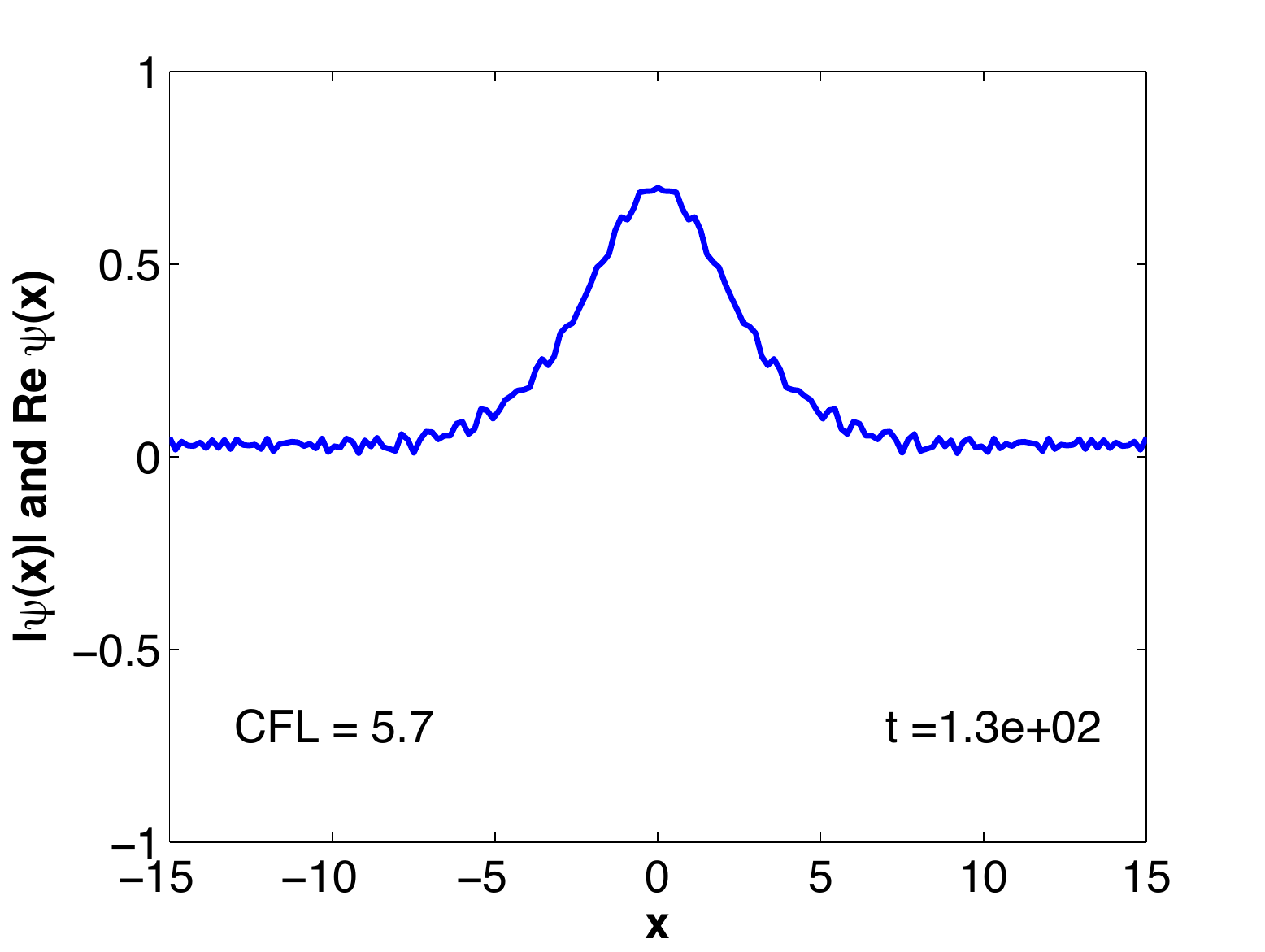}}}
\\
   \rotatebox{0}{\resizebox{!}{0.33\linewidth}{%
   \includegraphics{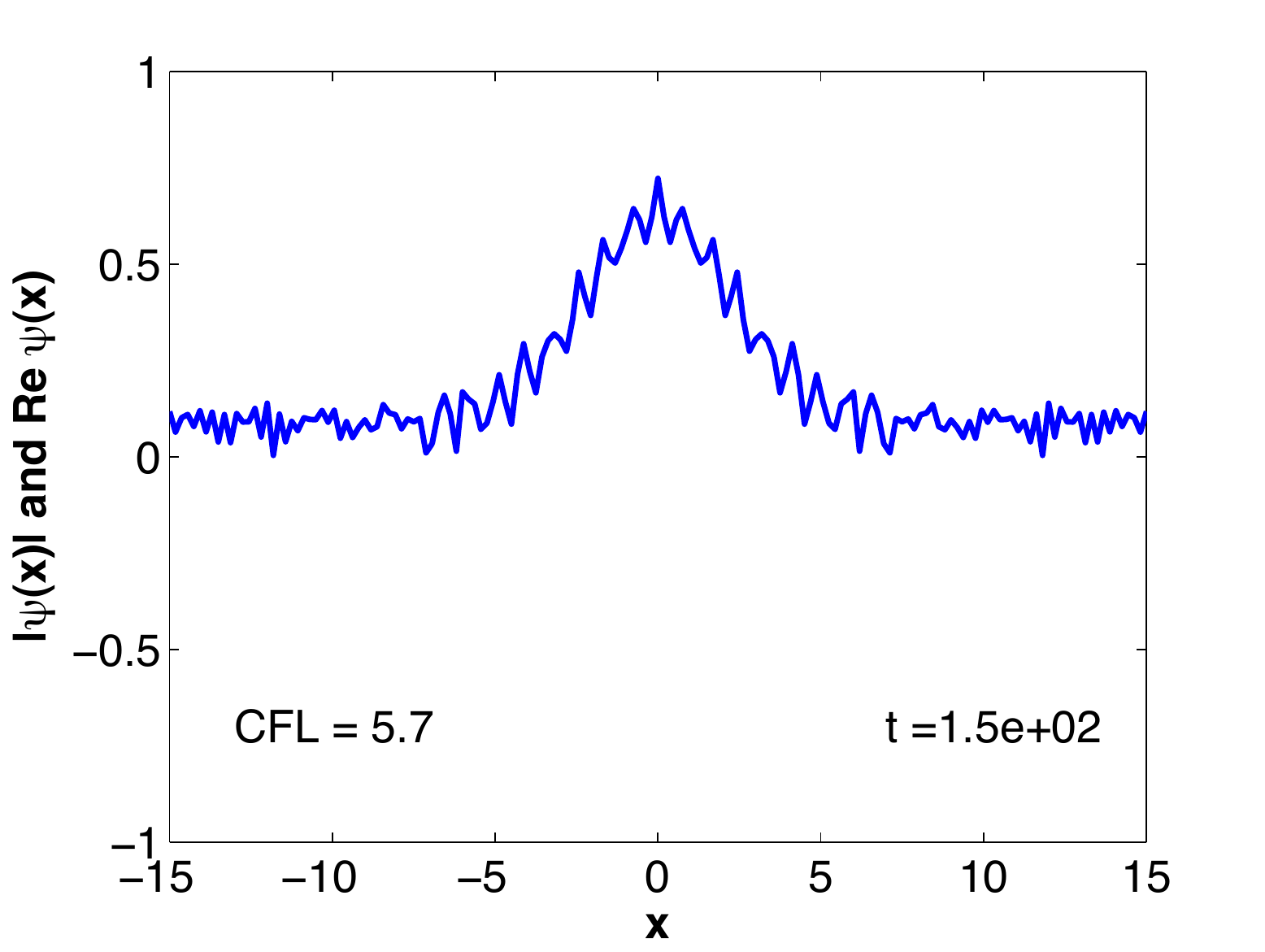}}}
      \rotatebox{0}{\resizebox{!}{0.33\linewidth}{%
   \includegraphics{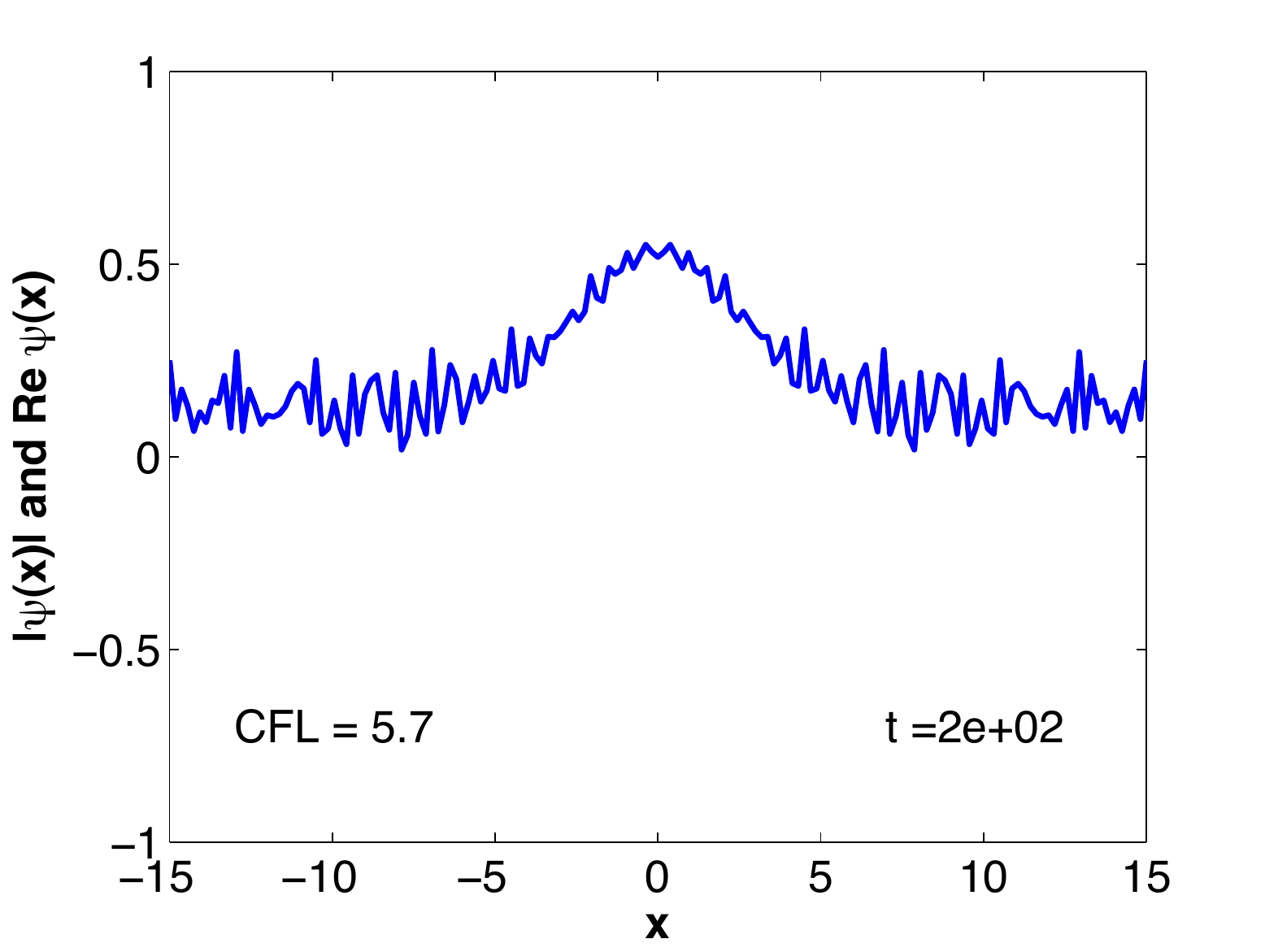}}}
\end{center}
\caption{Instability for $\tau/h^2 = 5.7$}
\label{fig4}
\end{figure}

In a second example, we take the same initial data and parameters $K = 80$ and $h = 0.1875$, except that we take a much smaller $\tau = 0.001$ making the CFL number equal to 0.028. However, we break artificially the symplecticity of the integrator by replacing the exact evaluation of the exponential in the flow $\Phi_A^\tau$ by its Taylor approximation of order 2: 
$$
\exp(\tau A) \simeq I + \tau A + \frac{\tau^2}{2} A^2.  
$$
\begin{figure}[ht]
\begin{center}
\rotatebox{0}{\resizebox{!}{0.33\linewidth}{%
   \includegraphics{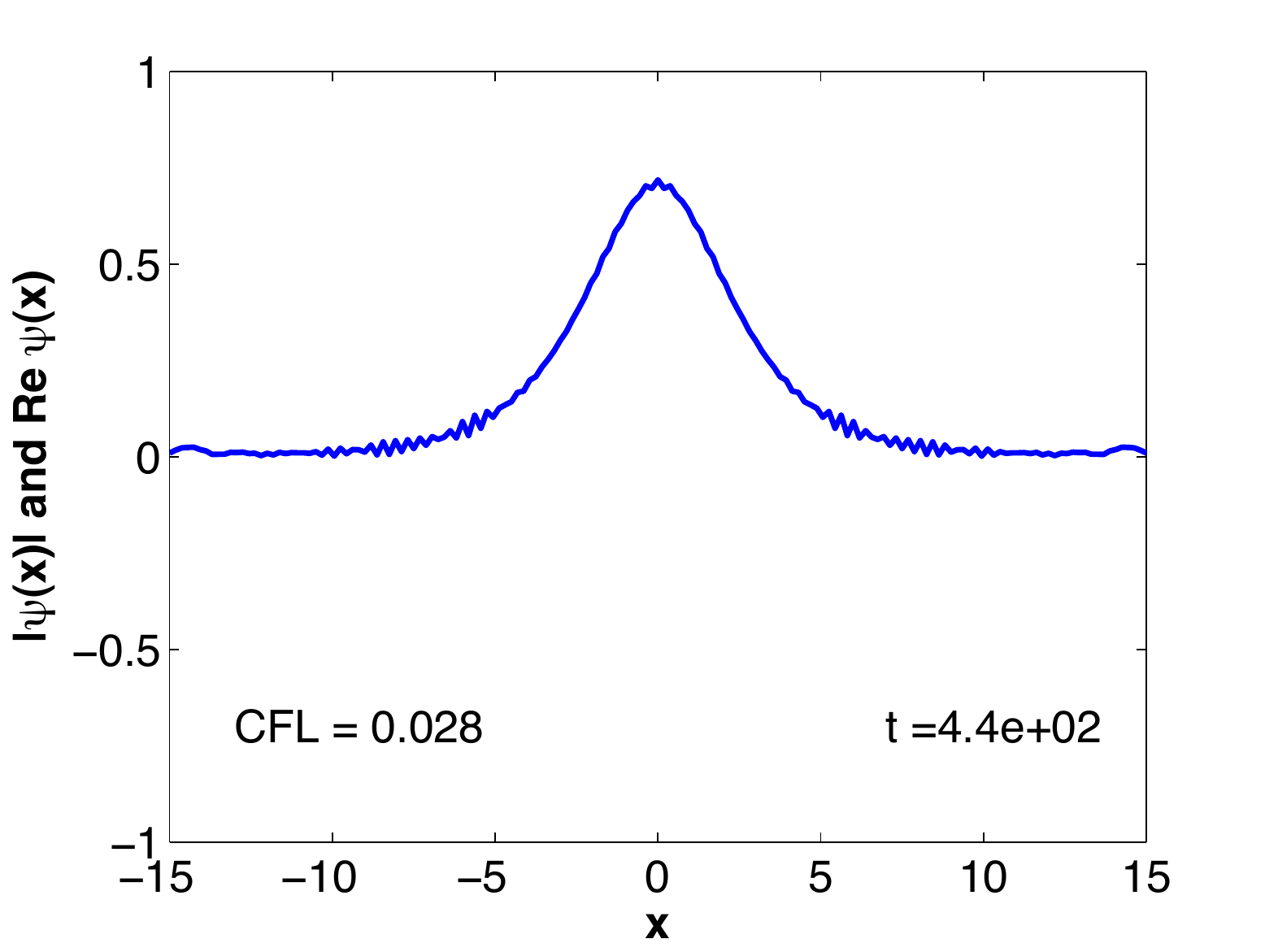}}}
   \rotatebox{0}{\resizebox{!}{0.33\linewidth}{%
   \includegraphics{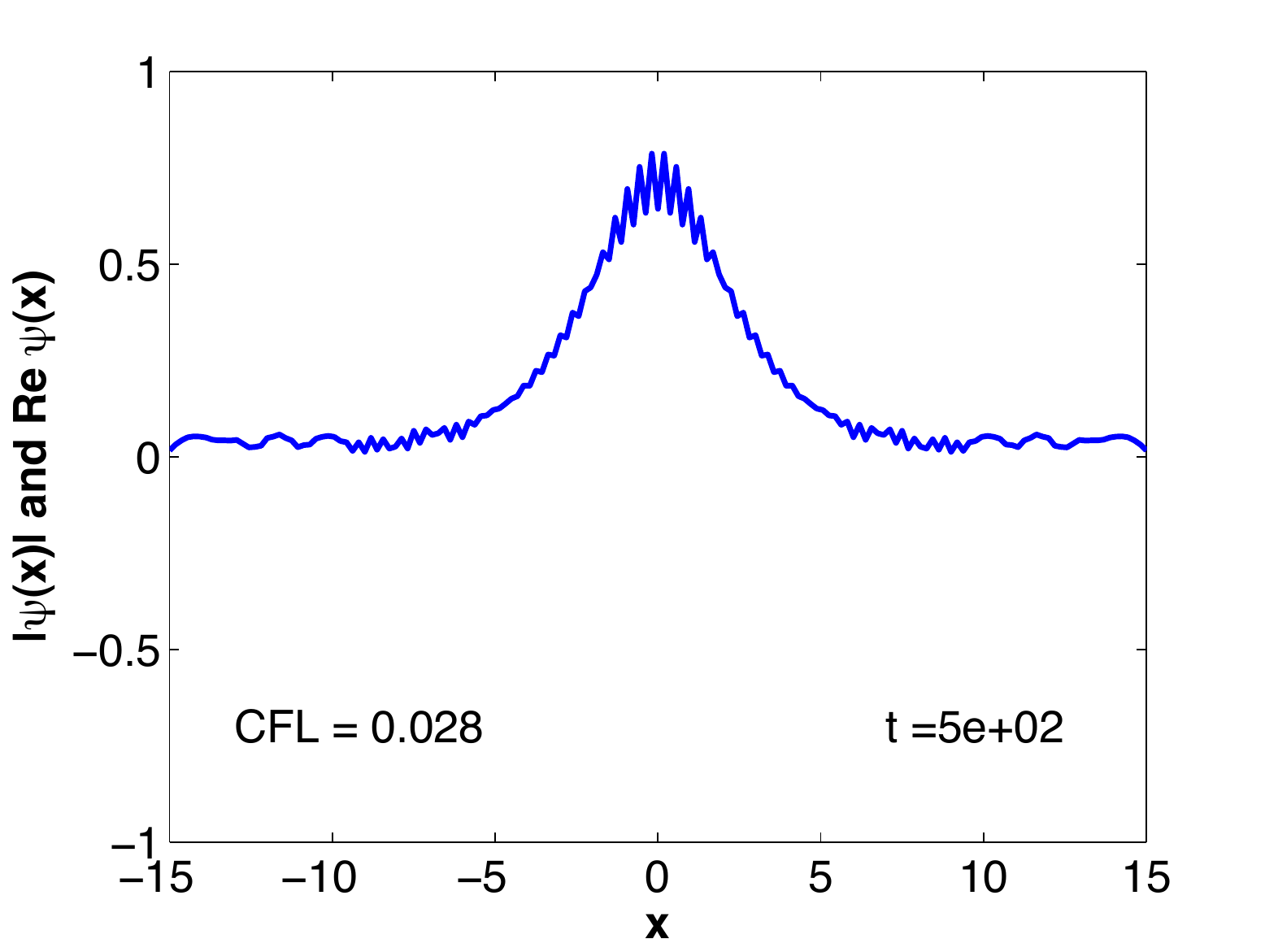}}} \\
\rotatebox{0}{\resizebox{!}{0.33\linewidth}{%
   \includegraphics{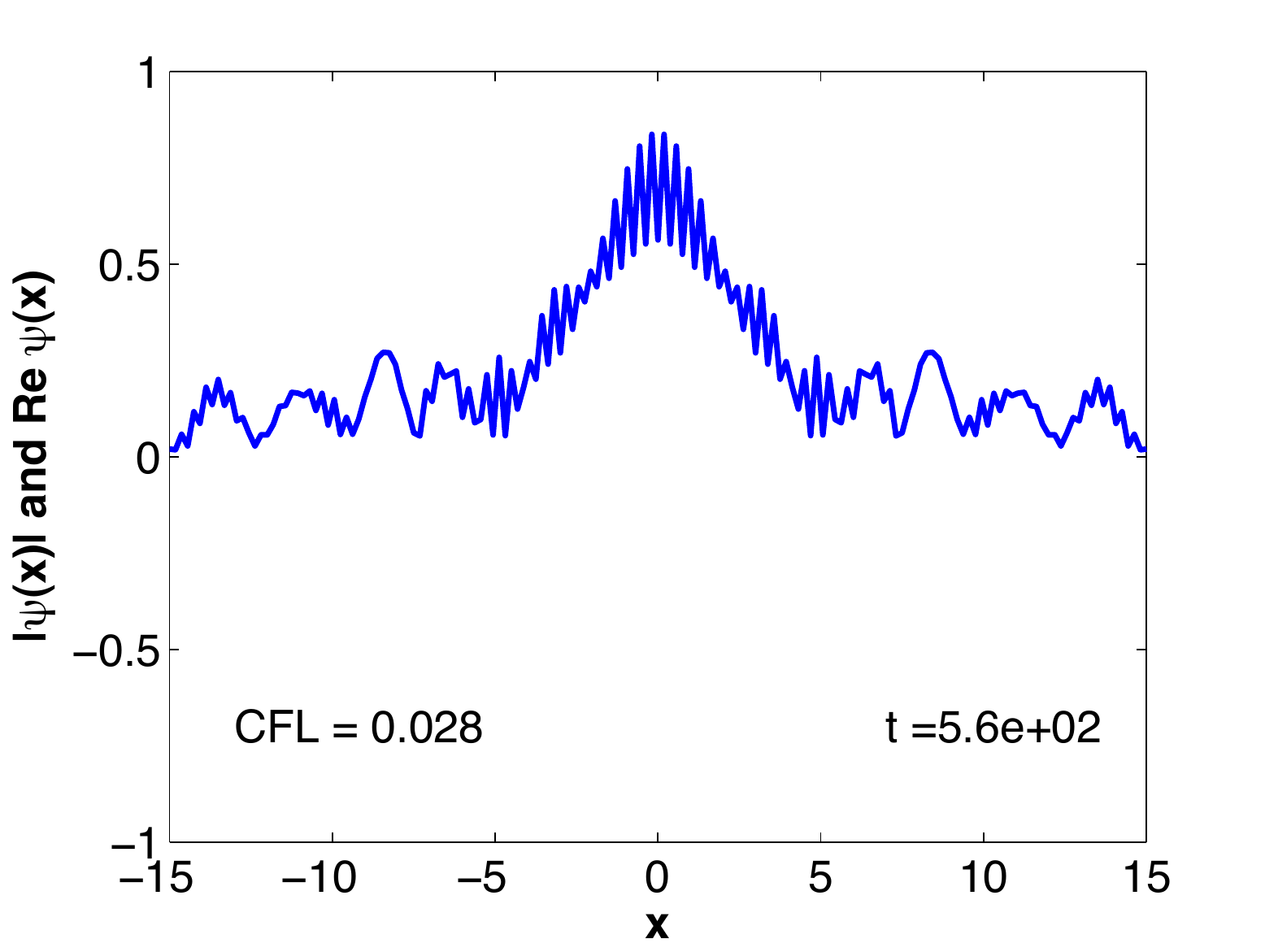}}}
   \rotatebox{0}{\resizebox{!}{0.33\linewidth}{%
   \includegraphics{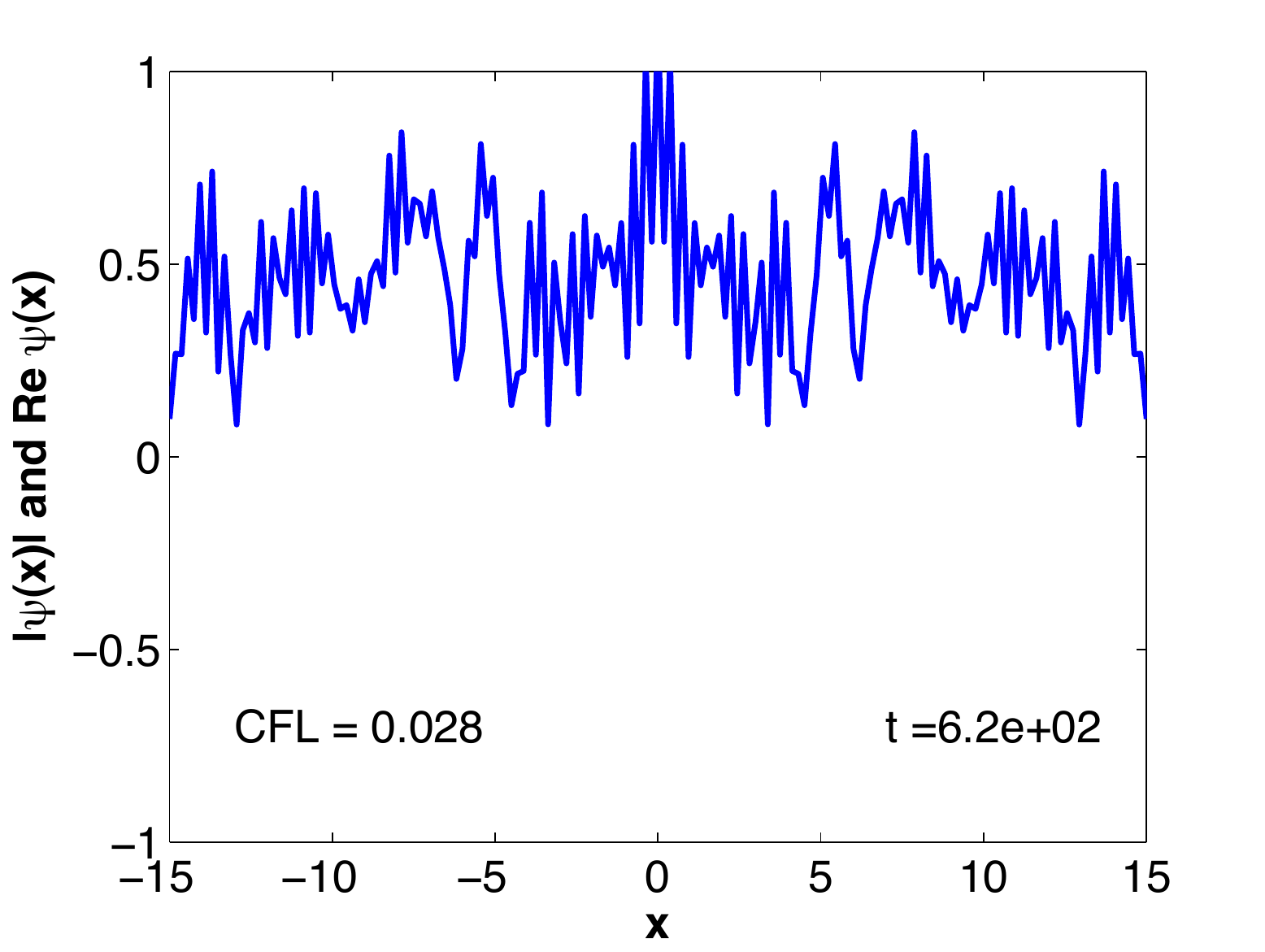}}}
\end{center}
\caption{Instability of non symplectic integrators}
\label{fig10}
\end{figure}
As before, we observe in Figure \ref{fig10} some instability phenomenon after some time, despite the fact that the CFL number is very small. Such an instability is due to the non symplectic nature of the integrator, which prevents the existence of a modified energy preserved by the numerical scheme.

Finally, we consider the same initial condition and numbers $K$ and $h$, but we take $\tau = 0.02$ making the CFL number be equal to 0.57 and we compute the exponential exactly making the scheme symplectic. 
\begin{figure}[ht]
\begin{center}
\rotatebox{0}{\resizebox{!}{0.33\linewidth}{%
   \includegraphics{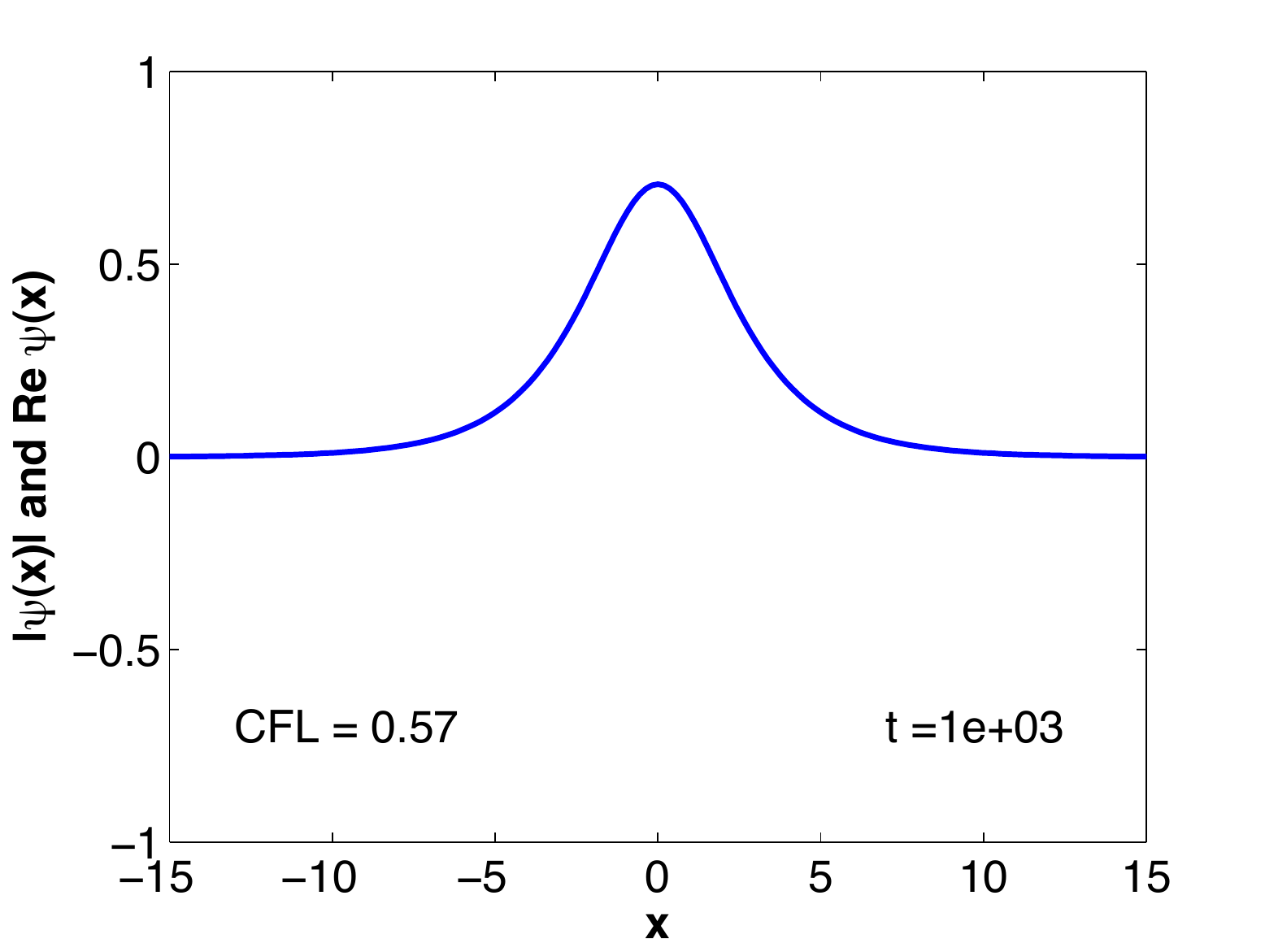}}}
   \rotatebox{0}{\resizebox{!}{0.33\linewidth}{%
   \includegraphics{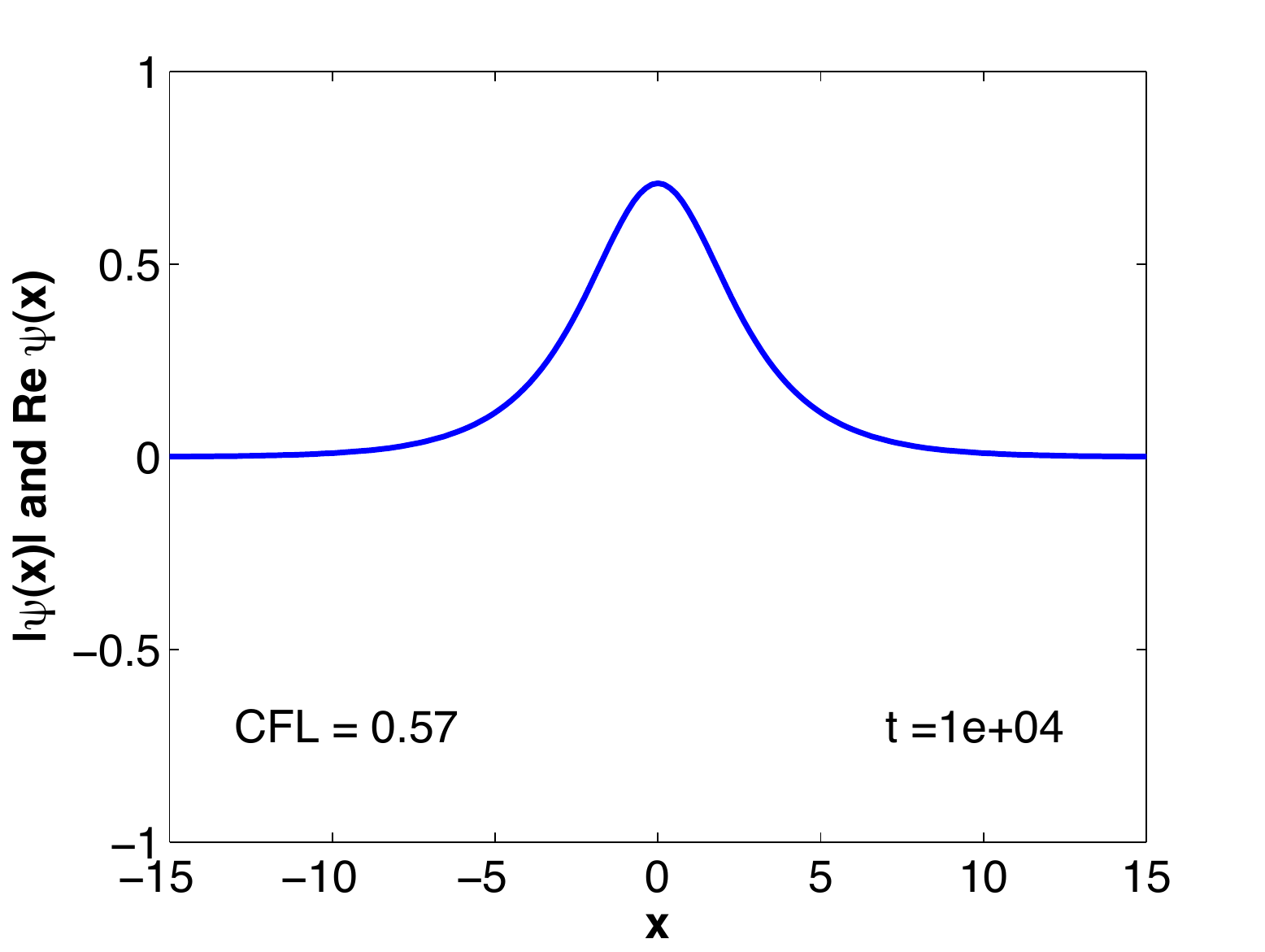}}} \\
\rotatebox{0}{\resizebox{!}{0.33\linewidth}{%
   \includegraphics{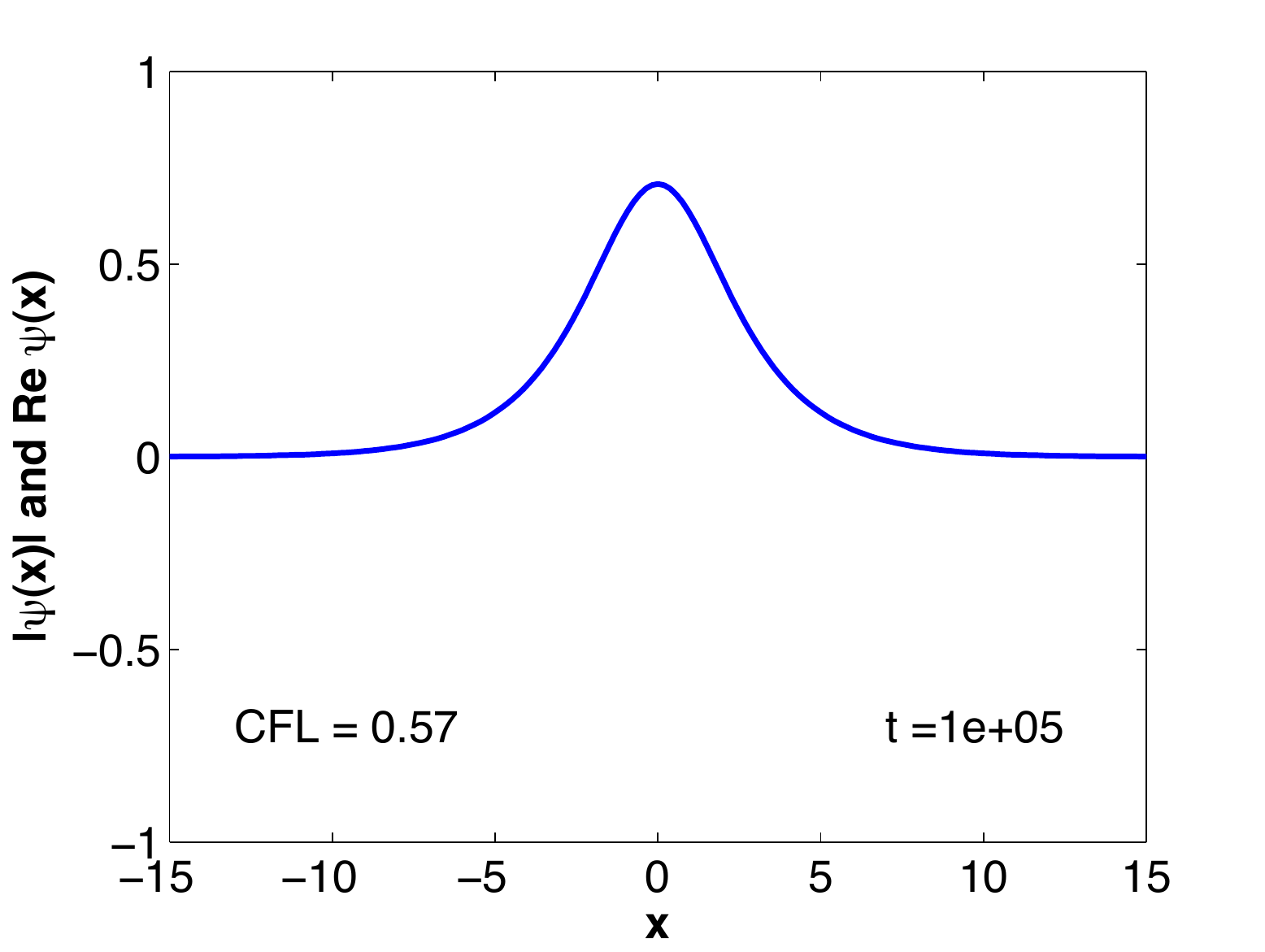}}} 
   \rotatebox{0}{\resizebox{!}{0.33\linewidth}{%
   \includegraphics{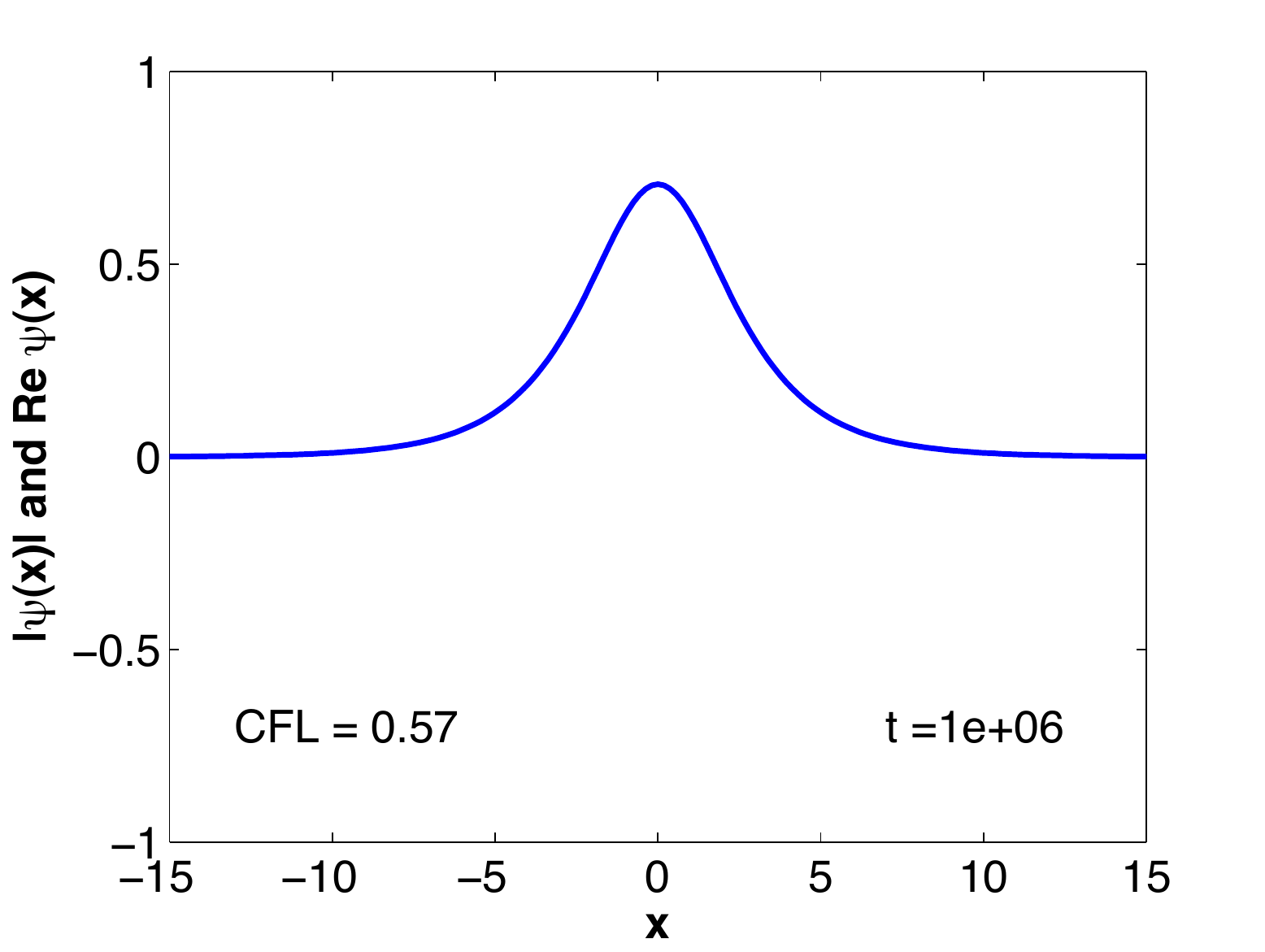}}}
\end{center}
\caption{Long time stability for $\tau/h^2 = 0.57$}
\label{fig6}
\end{figure}
In Figure \eqref{fig6} we can observe that the soliton is preserved for a very long time, up to $t = 10^6$ which corresponds to $2.10^8$ iterations. This result illustrates our  Theorem \ref{th:splitting}. 

\section{The continuous case}

Before giving the proofs of the Theorems presented above, we recall
here the main lines of the proof of the orbital stability result in
the continuous and symmetric case obtained first by \cite{Weinstein85}
(see also \cite{Grill87,Grill90, Frohlich04}). The proofs of the
discrete results will be essentially variations on the same theme.
The method is based on the variational characterization of the soliton
$\eta$ as the unique real symmetric minimizer of the problem
\begin{equation}
\label{min}
\min_{N(\psi) = 1} H(\psi)
\end{equation}
where  $H$ is the Hamiltonian \eqref{hc} and $N$ the norm \eqref{nc}. 
\begin{remark}
\label{lag}
By the method of Lagrange multipliers there exists $\lambda>0$ such that
$$
-\partial_{xx}\eta-\eta^3=-\lambda\eta.
$$
\end{remark}

\begin{remark}
We only consider the case where $N(\eta) = 1$ in order to avoid the introduction of a supplementary parameter. It is clear to the reader that we could also consider the numerical approximation of any given soliton, provided that its $L^2$ norm enters into all the constants appearing in the estimates below. 
\end{remark}

In the following, we set
$$
V = \{\; \psi \in H^1(\R;\C) \quad  | \quad \psi(-x) = \psi(x)\; \}. 
$$
We also define the real scalar product
$$
\langle \varphi, \psi \rangle = \mathrm{Re} \int_{\R} \varphi(x) \overline{\psi(x)} \dd x. 
$$
This scalar product allows to identify $H^1(\R;\C)$ with the product $H^1(\R;\R) \times H^1(\R;\R)$ as follows: 
If $\psi = \frac{1}{\sqrt{2}}(q + ip)$ and $\varphi = \frac{1}{\sqrt{2}}(q' + ip')$ where $p$, $q$, $p'$ and  $q'$ are real symmetric $H^1(\R;\R)$ functions, then we have 
$$
\langle \varphi, \psi \rangle = \frac12 \int_{\R} q(x)q'(x) + p(x)p'(x) \dd x. 
$$ 
The real scalar product on $H^1(\R;\C) \simeq H^1(\R;\R) \times H^1(\R;\R)$ is then given by 
$$
(\varphi, \psi)  = \langle \varphi, \psi \rangle + \langle \partial_x \varphi, \partial_x \psi \rangle, 
$$
and we set
$$
\Norm{\varphi}{H^1}^2 := (\varphi,\varphi) = \frac12 \int_{\R} |\partial_x p|^2 + |\partial_x q|^2  + |p|^2 + |q|^2 \dd x
$$
for $\varphi = \frac{q + ip}{\sqrt{2}}$. In the rest of this paper, we often amalgamate the two complex and real notations. 

In the following, we set 
\begin{equation}
\label{eq:UR}
\Uc(R) = \{ \varphi \in V\, | \, \dist( \psi, \Gamma) < R\}, 
\end{equation}
where $\Gamma$ is defined in \eqref{eq:defgamma}, and the distance is measured in $H^1$ norm.  

Note that the Hamiltonian function $H$ and the norm $N$ are smooth in $H^1$ (using the fact that $H^1$ is an algebra). Moreover, these functions are gauge invariant, in the sense that for all $\varphi \in H^1$ and all $\alpha\in \R$, we have $H(e^{i\alpha}\varphi) = H(\varphi)$ and $N(e^{i\alpha}\varphi) = N(\varphi)$. 
Due to this invariance, it is 
immediate to realize that the whole manifold $\Gamma$ is formed by
minima of the minimization problem \eqref{min}. Then it is well known \cite{Weinstein85,Grill87,Grill90,Frohlich04} that these minima are
nondegenenerate in the directions transversal to the orbit $\Gamma$ defined in \eqref{eq:defgamma}, for symmetric functions. 

%
More precisely, following \cite{Frohlich04}, we define the following
set of coordinates in the vicinity of $\Gamma$: set
\begin{equation}
\label{eq:W}
W = \{Êu \in VÊ\, |Ê\, \langle u ,\eta \rangle = \langle u ,i\eta \rangle =  0\},
\end{equation}
equipped with the $H^1$ norm induced by the space $V$.  As $i \eta$ is
tangent to the curve $\Gamma$ and orthogonal\footnote{Recall that here
$\langle\, \cdot \, , \, \cdot \, \rangle$ is a real scalar product.}
to $\eta$, the previous $W$ can be interpreted as the space orthogonal
to the plane containing the planar curve $\Gamma$. Note that $W$ is invariant under the multiplication by complex
number: for any $z \in \C$, if $u \in W$ then $zu \in W$.

We define the map $\chi$ as follows:  
\begin{equation}
\label{eq:chi}
 \T \times \R \times W \ni (\alpha,r,u) \mapsto \chi(\alpha,r,u) = e^{i\alpha}(  (1+r)\eta + u) \in V, 
\end{equation}
where $\T = \R \slash (2\pi \Z)$ is the one-dimensional torus. 
 
The following Lemma can be found in \cite[Section 5, Proposition 1]{Frohlich04}. In our symmetric situation, we give here an independent proof that will later be easily transfered to the situation of discrete systems: 
\begin{lemma}
\label{lem:1}
There exist  constants $r_0$ and $R$ such that   the application $\chi$ is smooth and bounded with bounded derivatives from $\T \times [-r_0,r_0] \times B(R)$ to $V$, and such for all $\varphi \in \Uc(R)$, there exists $(\alpha,r,u) \in \T \times \R \times W$ such that $\varphi = \chi(\alpha,r,u)$. Moreover, the application $\chi^{-1}$ is smooth with bounded derivatives on $\Uc(R)$, and there exists a constant $C$ such that for all $\psi \in \Uc(R)$, we have 
\begin{equation}
\label{eq:bdmerde}
\Norm{u(\psi)}{H^1} \leq C \dist(\psi, \Gamma). 
\end{equation}
\end{lemma}
\begin{proof}
The first part of this lemma is clear using the explicit formula for $\chi$. To prove the second one, let us consider the projection of $\psi$ onto the plane generated by $(\eta,i\eta)$:  
$$
\langle \psi, \eta \rangle \eta + \langle \psi, i \eta \rangle i \eta =: z(\psi) \eta
$$
with $z(\psi) = \langle \psi, \eta \rangle + i  \langle \psi, i \eta \rangle=\int \psi \bar\eta \in \C$. Note that the application $\psi \mapsto z(\psi)$ is  smooth with bounded derivatives from $V$ to $\C$. Moreover, we have 
$$
\dist(\psi,\Gamma)^2 \geq \inf_{\alpha} N(\psi - e^{i\alpha} \eta) \geq ||z(\psi)|^2 - 1|.
$$
Hence for $R\leq 1/2$ and for all $\psi\in \Uc(R)$,  we have $|z(\psi)| \in [1/2,3/2]$. 
This shows that the applications
$$
\Uc(R) \ni\psi \mapsto\hat \alpha(\psi) = \arg(z(\psi) ) \in \T
$$
and 
$$
\Uc(R) \ni\psi \mapsto \hat r(\psi) = |z(\psi)| - 1 \in [-1/2,1/2]
$$
are well defined and smooth with bounded derivatives on $\Uc(R)$ (as composition of smooth functions with bounded derivatives).  Moreover, we have $\psi - z(\psi) \eta \in W$: as $W$ is invariant under the  multiplication by complex numbers, the function 
$$
\hat u(\psi) :=  e^{-i \hat\alpha(\psi)} \psi - (1+\hat r(\psi))\eta = e^{-i\hat \alpha(\psi)} (\psi - z(\psi) \eta)
$$
is in $W$, smooth for $\psi \in \Uc(R)$, and satisfies $\psi = \chi(\hat \alpha(\psi), \hat r(\psi), \hat u(\psi))$.  

To prove \eqref{eq:bdmerde} let $\psi^* \in \Gamma$ be the element of $\Gamma$ realizing the minimum in the right-hand side (which exists by compactness of $\Gamma$). As $\psi^* \in \Gamma$ we have $\hat u(\psi^*) = 0$.  As the fonction $\psi \mapsto \hat u(\psi)$ is uniformly Lipschitz in $\Uc(R)$, we have 
$$
\Norm{\hat u(\psi)}{H^1}Ê\leq C \Norm{\psi - \psi^*}{H^1} = C \dist(\psi, \Gamma), 
$$
which gives the result. 
\end{proof}

Let us now define the function $u \mapsto r(u)$ from $W$ to $\R$ by the implicit relation 
$$
N(\chi(\alpha,r(u), u)) = 1. 
$$
 By explicit calculation, we have 
\begin{equation}
\label{eq:rNu}
r(u) = -1 + \sqrt{1 - N(u)},
\end{equation}
from which we deduce that $r(u)$ is well defined and smooth in a
 neighborhood of $0$ in $H^1$, and moreover that $\Norm{r(u)}{H^1}
 = \mathcal{O}(\Norm{u}{H^1}^2)$ if $u$ is sufficiently small.  Hence,
 $(\alpha,u) \mapsto \chi(\alpha,r(u), u)$ is a local parametrization
 of $\Sc$ in a neighborhood of $\Gamma \subset \Sc$, where
\begin{equation}
\label{s}
\Sc:=\left\{\psi\in V\, | \, N(\psi)=1  \right\}. 
\end{equation}
Now let us define the function
\begin{equation}
\label{eq:defHc}
\Hc(u) = H(\chi(\alpha,r(u),u)), 
\end{equation}
which is well defined on $W$ by gauge invariance of $H$. Moreover, this function is smooth in a neighborhood of $0$. 
Then it can be shown (see \cite{Frohlich04}) that $u = 0$ is a non degenerate minimum of $\Hc(u)$: we have 
$$
\dd \Hc (0) = 0, \quad \mbox{and}\quad  \forall\, U \in W, \quad \dd^2 \Hc(0)(U,U) \geq c \Norm{U}{H^1}^2. 
$$
Note that as $\Hc$ is smooth with locally bounded derivatives, the last coercivity estimate extends to a neighborhood of $0$ uniformly: there exist positive constants $c$ and $\rho$ such that 
\begin{equation}
\label{eq:coerci}
\forall\, u \in B(\rho), \quad\forall\, U \in W,  \quad \dd^2 \Hc(u)(U,U) \geq c \Norm{U}{H^1}^2, 
\end{equation}
where $B(\rho)$ denotes the ball of radius $\rho$ in $W$. In other words, the function $\Hc$ is strictly convex on $B(\rho)$ and has a strict minimum at $u = 0$. 

With these results at hand, let $\psi \in \Sc$, and assume that $\dist(\psi,\Gamma)$ is small enough so that we can write
$$
\psi = e^{i\alpha}( (1 + r(u))\eta + u),  
$$
for some $(\alpha,u) \in \T \times W$. 
Then for some constant $C$ an sufficiently small $u$, we have 
$$
\dist(\psi,\Gamma) \leq \Norm{\psi - e^{i\alpha}\eta}{H^1} \leq  C ( r(u) + \Norm{u}{H^1}) \leq C \Norm{u}{H^1}. 
$$
Now as $u = 0$ is a minimum of the strictly convex function $\Hc$ on the ball $B(\rho)$, we can write 
$$
H(\psi) - H(\eta) = \Hc(u) - \Hc(0) > \gamma \Norm{u}{H^1}^2 >  c \dist(\psi,\Gamma)^2
$$
for some constants $\gamma$ and $c > 0$ depending only on $\rho$. Then a Taylor expansion of $\Hc$ around $u = 0$ shows that 
$$
| \Hc(u) - \Hc(0) | \leq C \Norm{u}{H^1}^2,
$$
for some constant $C$ depending on $\rho$ and $H$ but not on $u \in B(\rho)$. Hence using \eqref{eq:bdmerde} we obtain 
the existence of constants $c$, $C$ and $R_0 > 0$ such that for all $\psi \in \Sc$ such that $\dist(\psi,\Gamma) < R_0$, we have 
$$
c\dist(\psi,\Gamma)^2 \leq |ÊH(\psi) - H(\eta)| \leq C \dist(\psi,\Gamma)^2. 
$$

The stability result \eqref{eq:bite} is then an easy consequence of
this relation: Assume that $\psi_0 \in \Sc$ satifies
$\dist(\psi(0),\Gamma)\leq \delta < \delta_0$ where $\delta_0 < R_0$,
and let $\psi(t)$, $t >0$ be the solution of \eqref{nls} starting at
$\psi(0)\equiv \psi_0$.  Then by preservation of the energy $H$ and
norm $N$, we have $\psi(t) \in \Sc$ for all $t > 0$, and moreover as
long as $\psi(t)$ is such that $\dist(\psi(t),\Gamma) < R_0$ we can
write
\begin{equation}
\label{eq:aslongas}
c\dist(\psi(t),\Gamma)^2 \leq |ÊH(\psi(t)) - H(\eta)| = |ÊH(\psi(0)) - H(\eta)| \leq C \dist(\psi(0),\Gamma)^2. 
\end{equation}
Hence if $\delta_0$ is small enough, this shows that for all $t$, $\dist(\psi(t),\Gamma) < R_0$ and that \eqref{eq:aslongas} is in fact valid for all times $t > 0$. 
This implies \eqref{eq:bite} in the case $N(\psi)=1$. 

\section{An abstract result}

In this section, we prove an abstract result for the existence and stability of discrete solitons. We first give conditions ensuring that a discrete Hamiltonian acting on a discrete subspace of $H^1$ possesses a minimizing soliton. We then show how the existence of a discrete flow (almost) preserving the Hamiltonian and the $L^2$ norm ensures the numerical orbital stability over long times. In the next sections, we will apply this result to the three levels of discretization described above.  

\subsection{Approximate problems}
We consider a set of parameter $\Sigma\in \R^p$ and a function $\epsilon:\Sigma \to \R^+$. This function will measure the ``distance" between the discrete  and continuous problems. 

For all $ \mu \in \Sigma$, we consider a Hilbert space $V_ \mu$ equipped with a norm $\Norm{\cdot}{\mu}$. For a given number $R$, we denote by $B_\mu(R)$ the ball of radius $R$ in $V_\mu$. Moreover, for a given $k \geq 0$ a function $F: V_\mu \to \C$ of class $\mathcal{C}^k$, and a given $\psi_\mu \in V_\mu$, we set for all $n = 0,\ldots,k$ 
$$
\Norm{\dd^n F(\psi_\mu)}{\mu}  = \sup_{U^1,\ldots,U^n \in V_\mu\backslash{\{0\}}} \frac{|\dd^n F(\psi_\mu)(U^1,\ldots,U^n)|}{\Norm{U^1}{\mu} \ldots \Norm{U^n}{\mu}}
$$
and we set 
$$
\Norm{F}{\mathcal{C}^k(B_\mu(R))} = \sup_{n = 0,\ldots,k}\,  \sup_{\psi_\mu \in B_\mu(R)} \Norm{\dd^n F(\psi_\mu)}{\mu}. 
$$
Moreover, we say that $F$ is {\em gauge invariant}Ê if it satisfies, for all $\alpha\in \T$ and all $\psi_\mu \in V_\mu$, $F(e^{i\alpha} \psi_\mu) = F(\psi_\mu)$. Similarly, we say that $G: V_\mu\times V_\mu \to \C$ is gauge invariant if for all $\varphi_\mu$ and $\psi_\mu$ in $V_\mu$, and all $\alpha \in \T$, we have $G(e^{i\alpha} \varphi_\mu, e^{i\alpha} \psi_\mu) = G(\varphi_\mu,\psi_\mu)$. 

We assume that the family $(V_\mu)_{\mu \in \Sigma}$ satisfies the following assumptions:  

\begin{itemize}
\item[\textbf{(i)}] For all $\mu \in \Sigma$, there exist a linear embedding $i_ \mu : V_ \mu\to H^1$ and a projection $\pi_ \mu: H^1 \mapsto V_ \mu$ that are gauge invariant in the sense that for all $\alpha \in \T$ and $\psi_\mu \in V_ \mu$, $e^{i\alpha} i_ \mu \psi_ \mu = i_ \mu e^{i\alpha} \psi_ \mu$ and for all $\psi \in V$,  $e^{i\alpha} \pi_ \mu \psi = \pi_ \mu e^{i\alpha} \psi$. Morever, we assume that $i_ \mu$ and $\pi_ \mu$ are real in the sense that $\overline {i_\mu \psi_\mu} = i_\mu  \overline \psi$ and $\overline {\pi_\mu \psi_\mu} = \pi_\mu  \overline \psi$, and that they satisfy the relation $\pi_\mu \circ i_\mu = \mathrm{id}\hspace{-0,15cm}\mid_{V_\mu}$. Finally, we assume that there exists a constant $R_0 > 1$ such that for all $\mu \in \Sigma$, and $\varphi_\mu \in B_\mu(R_0)$, 
$$
\big|\Norm{\varphi_\mu}{\mu}^2 - \Norm{i_\mu \varphi_\mu}{H^1}^2 \big| \leq \epsilon(\mu)\Norm{i_\mu \varphi_\mu}{H^1}^2.
$$

\item[\textbf{(ii)}] 
For all $\mu \in \Sigma$, there exists a gauge invariant real scalar product $\langle \, \cdot\, , \, \cdot\,  \rangle_\mu$ such that setting $N_\mu(\psi_\mu) = \langle \psi_\mu, \psi_\mu \rangle_\mu$,  we have $N_\mu(\psi_\mu) \leq \Norm{\psi_\mu}{\mu}^2$ and 
$$
\Norm{N \circ i_ \mu - N_ \mu}{\mathcal{C}^2(B_ \mu(R_0))} \leq \epsilon(\mu).  
$$
\item[\textbf{(iii)}]
For all $ \mu \in \Sigma$, there exists a gauge invariant function $H_ \mu: V_ \mu \to \R$ which is a modified Hamiltonian in the sense that
$$
\Norm{H \circ i_ \mu - H_ \mu}{\mathcal{C}^2(B_ \mu(R_0))} \leq \epsilon(\mu). 
$$

\item[\textbf{(iv)}]
If $\eta$ is the continuous soliton \eqref{phis} defined in the previous section, we have for all $\mu \in \Sigma$ 
\begin{equation}
\label{eq:approxeta}
\Norm{i_ \mu\pi_ \mu \eta - \eta}{H^1} \leq \epsilon( \mu). 
\end{equation}

\end{itemize}

\medskip 

Note that using \textbf{(i)}, there exist constants $c$, $C$ and
$\epsilon_0$ such that for $\psi_\mu \in V_\mu$ and $\mu \in \Sigma$
such that $\epsilon(\mu) < \epsilon_0$, we have
\begin{equation}
\label{eq:normmu}
c\Norm{i_ \mu \psi_ \mu}{H^1} \leq \Norm{\psi_ \mu}{ \mu} \leq C\Norm{i_ \mu \psi_ \mu}{H^1}. 
\end{equation}

In the rest of this Section, we will assume that the hypothesis \textbf{(i)--(iv)} are satisfied.

\subsection{Local coordinate system}

We will assume here that all the $\mu \in \Sigma$ considered satisfy the relation $\epsilon(\mu) < \epsilon_0$ for some constant $\epsilon_0$ to be precised along the text. 
In echo to \eqref{s} we define for all $\mu \in \Sigma$
$$
\Sc_\mu = \{ \psi_\mu \in V_\mu \, | \, N_\mu(\psi_\mu) = 1\}, 
$$
and the tangent space to $\pi_\mu \eta$ (compare \eqref{eq:W}): 
$$
W_\mu = \{Êu_\mu \in V_\mu \, | \, \langle u_\mu, \pi_\mu \eta \rangle_\mu =  \langle u_\mu, i \pi_\mu \eta \rangle_\mu = 0 \}. 
$$
Note that $i_\mu W_\mu$ is not included in $W$. 

By a slight abuse of notation, we will write $u_\mu \in B_\mu(\gamma)$ the ball of radius $\gamma$ in $W_\mu$ (instead of $B_\mu(\gamma) \cap W_\mu$) for $\gamma > 0$. We also set for $R > 0$ (compare \eqref{eq:UR})
\begin{equation}
\label{eq:URmu}
\{Ê\psi_\mu \in V_\mu\, | \, \dist_\mu(\psi_\mu,\pi_\mu \Gamma) \leq \gamma\}, 
\end{equation}
where $\dist_\mu$ denotes the distance measured in the norm $\Norm{\cdot}{\mu}$ and where 
$$
\pi_\mu \Gamma := \bigcup_{\alpha \in \R} \{ e^{i\alpha} \pi_\eta \eta\}. 
$$

We then define the discrete application $\chi_\mu$ (see \eqref{eq:chi}):
$$
\T \times \R \times W_\mu \ni (\alpha,r,u_\mu) \mapsto \chi_\mu(\alpha,r,u_\mu) = e^{i\alpha}(  (1+r)\pi_\mu\eta + u_\mu) \in V_\mu.
$$
\begin{lemma}
\label{lem:2}
There exist constants $\epsilon_0$, $r_0$, $C$ and $R$ such that for
all $\mu \in \Sigma$ with $\epsilon(\mu) < \epsilon_0$, the
application $\chi_\mu$ is smooth and bounded with uniformly bounded
derivatives (with respect to $\mu$) from $\T \times [-r_0,r_0] \times
B_\mu(R)$ to $V$, and such for all $\varphi_\mu \in \Uc_\mu(R)$, there
exists $(\alpha,r,u_\mu) \in \T \times \R \times W$ such that
$\varphi_\mu = \chi_\mu(\alpha,r,u_\mu)$. Moreover, the application
$\chi_\mu^{-1}$ is smooth with uniformly bounded derivatives on
$\Uc_\mu(R)$, and for all $\psi_\mu \in \Uc_\mu(R)$, we have
\begin{equation}
\label{eq:bdmerde2}
\Norm{u_\mu(\psi_\mu)}{H^1} \leq C \dist_\mu(\psi_\mu, \pi_\mu\Gamma). 
\end{equation}
\end{lemma}
\begin{proof}
The proof is exactly the same as the one of Lemma \ref{lem:1} by replacing $\langle\, \cdot\, , \, \cdot \, \rangle$ by $\langle\, \cdot\, , \, \cdot \, \rangle_\mu$, $N$ by $N_\mu$ and $\eta$ by $\pi_\mu \eta$. The fact that the constants are uniform in $\mu$ is a consequence of the direct construction made in the proof of this Lemma and of the hypothesis  $\textbf{(i)-(iv)}$. Note that we use the fact that 
\begin{equation}
\label{eq:Nmu1}
|N_ \mu(\pi_ \mu \eta) - 1 |\leq C \epsilon(\mu), 
\end{equation}
for some constant $C$ independent on $\mu$, which is a consequence of 
\textbf{(ii)} and \eqref{eq:approxeta}, 
provided $\epsilon(\mu) < \epsilon_0$ is small enough to ensure that $\Norm{\pi_\mu\eta}{\mu} < R_1$ (which is possible upon using \eqref{eq:approxeta} and \eqref{eq:normmu}). 
\end{proof}

Note that using the gauge invariance of $i_\mu$, we have for all $(\alpha,r,u_\mu) \in \T \times \R \times W_\mu$
$$
i_\mu \chi_\mu(\alpha,r,u_\mu) - \chi(\alpha,r,i_\mu u_\mu) =
e^{i\alpha}(1+r)( i_\mu \pi_\mu \eta - \eta)
$$
and hence for all $u_\mu \in W_\mu$, and $r \in \R$, 
\begin{equation}
\label{eq:ichi}
\Norm{i_\mu \chi_\mu(\alpha,r,u_\mu) - \chi(\alpha,r,i_\mu u_\mu)}{H^1} \leq (1 + |r|) \epsilon(\mu). 
\end{equation}

%
%
%
%
%
%

Following the formalism of the previous section, we define for all $\mu \in \Sigma$ the function $u_\mu \mapsto r_\mu(u_\mu)$ on $W_\mu$ by the implicit relation  
$$
N_\mu(\chi_\mu(\alpha,r_\mu(u_\mu),u_\mu)) = 1, 
$$
so that $(\alpha,u_\mu)$ is a local coordinate system close to a
rescaling of $\pi_\mu\Gamma$.  Using the definition of $N_\mu$ and
$\chi_\mu$, we immediately  obtain that
$$
r_\mu(u_ \mu) = -1 + \sqrt{1 - \frac{N_ \mu(u_ \mu)}{N_ \mu(\pi_ \mu \eta)}}. 
$$
With this explicit expression, and using again \textbf{(ii)} and \eqref{eq:ichi}
there exist constants $\rho_0$, $C$ and $\epsilon_0$ such that for all $\mu \in \Sigma$ with $\epsilon(\mu) < \epsilon_0$,  $r_\mu$ is $\mathcal{C}^2(B_\mu(\rho_0))$, and 
\be \label{estimr}
\Norm{r_\mu - r\circ i_\mu}{\mathcal{C}^2(B_\mu(\rho_0))} \leq C \epsilon(\mu), 
\ee
where the function $r$ is defined in \eqref{eq:rNu}. 
Now defining (compare \eqref{eq:defHc})
$$
\Hc_\mu(u_ \mu) := H_\mu(\chi_\mu(\alpha, r_\mu(u_\mu), u_\mu)), 
$$
the previous relations, together with \textbf{(iii)} and \eqref{eq:ichi} imply that if $\rho_0$ is sufficiently small, $\Hc_\mu$ is well defined on $B_\mu(\rho_0)$, and moreover
\begin{equation}
\label{eq:approxHcal}
\Norm{\Hc \circ i_\mu - \Hc_\mu}{\mathcal{C}^2(B_ \mu(\rho_0))} \leq C \epsilon(\mu). 
\end{equation}
for some constant $C$ independent of $\mu$, and for all $\mu \in \Sigma$ such that $\epsilon(\mu) < \epsilon_0$.

\subsection{Existence of a discrete soliton}
In the previous section, we have shown that the continuous function $\Hc$ can be approximated by a function $\Hc_\mu$ on balls of fixed radius $\rho_0$ in $V_\delta$. This is the key argument to prove the following result: 
\begin{theorem}
Under the previous hypothesis, there exists $\epsilon_0$ such that for all $ \mu \in \Sigma$ with $\epsilon( \mu) \leq \epsilon_0$, there exists a discrete soliton $\eta_ \mu \in V_ \mu$  that realizes the minimum of 
$H_ \mu$ under the constraint $N_ \mu(\psi_ \mu) = 1$,
and such that 
\be\label{5.7bis}
\Norm{\eta_ \mu - \pi_ \mu \eta}{V_ \mu} \leq \epsilon( \mu).
\ee
Moreover, there exist constants $C$, $\delta_0$ and $\gamma_0$ such that for all $\mu \in \Sigma$ with $\epsilon(\mu) < \epsilon_0$, and all $\delta < \delta_0$, 
\be
\label{eq:ctrl}
\dist( i_ \mu \psi_ \mu,\Gamma)^2  \leq  C( |H_ \mu(\psi_ \mu) - H_ \mu(\eta_ \mu)| + \epsilon(\mu) + \delta ), 
\ee
for all  $\psi_\mu$ such that $\dist( i_ \mu \psi_ \mu,\Gamma) \leq \gamma_0$ and $|N_\mu(\psi_\mu) - 1| \leq \delta$. 
\end{theorem}
\begin{proof}
Let us take $\epsilon_0$ and $\rho_0$ as in the previous section. Recall that as $\eta$ is a minimizer of the continuous Hamiltonian $H$, and by definition of $\Hc$, we have $\dd \Hc(0) = 0$. 
Using \eqref{eq:approxHcal}, we deduce that for all $\mu \in \Sigma$ such that $\epsilon(\mu) < \epsilon_0$, 
\begin{equation}
\label{eq:dHc}
\Norm{\dd \Hc_\mu(0)}{\mu} \leq C \epsilon( \mu). 
\end{equation}
Moreover, for all $U \in W_\mu$, and $u_\mu \in B_\mu(\rho_0)$, we have using again \eqref{eq:approxHcal}
$$
| \dd^2 \Hc_\mu(u_\mu)(U,U) - \dd^2 \Hc(i_\mu u_\mu)(i_\mu U,i_\mu U)| \leq C \epsilon(\mu) \Norm{U}{\mu}^2 . 
$$
Using \eqref{eq:coerci} and \eqref{eq:normmu}, this shows that $\Hc_\mu$ is uniformly strictly convex in $B_\mu(\rho_0)$, i.e. satisfies 
$$
\forall\, u_\mu \in B_\mu(\rho_0), \quad 
\forall\, U \in W_\mu, \quad 
\dd^2 \Hc_\mu(u_\mu)(U,U) \geq c_0 \Norm{U}{\mu}^2 , 
$$
with a constant $c_0$ independent on $\mu$ such that $\epsilon(\mu) < \epsilon_0$ small enough. 

As $\Hc_\mu$ is strictly convex on the closed  ball $\overline{B_\mu}(\rho_0)$,  $\Hc_\mu$ reaches its minimum on $\overline{B_\mu}(\rho_0)$ at some point $u_\mu^*\in \overline{B_\mu}(\rho_0)$ (see for instance \cite{Ciarlet}). We want to prove that the minimum is reached in the interior of the ball. So assume on the contrary
that  $u_\mu^*$ is such that $\Norm{u_\mu^*}{\mu} = \rho_0$, then we have 
$$
\Hc_\mu(u_\mu^*) - \Hc_\mu(0) = \dd \Hc_\mu(0) \cdot u_\mu^* + h(u_\mu^*)
$$
with $h(u_\mu^*) > c_0 \Norm{u_\mu^*}{\mu}^2$. Hence, as $|\dd \Hc_\mu(0) \cdot u_\mu^*| \leq C \epsilon(\mu) \Norm{u_\mu^*}{\mu}$ (see \eqref{eq:dHc}) we get 
$$
\Hc_\mu(u_\mu^*) - \Hc_\mu(0) > c_0 \rho_0^2 - C\epsilon(\mu) \rho_0. 
$$
This shows that for $\epsilon_0$ sufficienly small, $\Hc_\mu(u_\mu^*) > \Hc_\mu(0)$ which is a contradiction. Hence the  $u_\mu^*$ is in the open ball $B_\mu(\rho_0)$ and thus 
$$
\dd \Hc_\mu(u^*_\mu) = 0. 
$$
Moreover, as $\Hc_\mu$ is uniformly convex on the ball $B_\mu(\rho_0)$, we have 
$$
\Norm{u^*}{\mu} \leq C \Norm{\dd \Hc_\mu(u^*) - \dd \Hc_\mu(0) }{\mu} \leq C \epsilon(\mu). 
$$
for some constant $C$ independent on $\mu$. 
Then setting 
\be\label{estimu}
\eta_ \mu := \chi_ \mu(0,r_ \mu(u_ \mu^*),u_ \mu^*)=(1+r_\mu(u_\mu^*))\pi_ \mu \eta+u_ \mu^*, 
\ee
we verify using \eqref{estimr} and \eqref{estimu} that we  have $\Norm{\pi_ \mu \eta - \eta_ \mu}{ \mu} \leq C\epsilon( \mu)$ for some constant $C$ independent on $\mu$. 

It remains to prove \eqref{eq:ctrl}. Let $\psi_ \mu \in V_\mu$ and $\alpha \in \T$, we have 
\begin{eqnarray*}
\Norm{i_\mu \psi_\mu - e^{i\alpha} \eta}{H^1} &\leq& \Norm{i_\mu \psi_\mu - e^{i\alpha} i_\mu \pi_\mu\eta }{H^1}Ê+ \Norm{i_\mu \pi_\mu \eta  -  \eta }{H^1}\\
&\leq& C\Norm{\psi_ \mu - e^{i\alpha} \pi_\mu\eta  }{\mu} + C\epsilon( \mu), 
\end{eqnarray*}
where we used \eqref{eq:normmu}. Hence we have for all $\psi_\mu$ 
\begin{equation}
\label{eq:ineq1}
\dist( i_\mu \psi_\mu,\Gamma) \leq C\dist_\mu( \psi_\mu, \pi_\mu\Gamma) + C \epsilon(\mu)
\end{equation}
for some constant independent of $\mu$. 
Similarly we prove that 
\begin{equation}
\label{eq:ineq2}
\dist_\mu( \psi_\mu, \pi_\mu\Gamma) \leq C \dist( i_\mu \psi_\mu,\Gamma) + C \epsilon(\mu),
\end{equation}
for some constant $C$ independent on $\mu$. 
Now let  $\psi_\mu$ be a function such that $\dist( i_\mu \psi_\mu,\Gamma) < \gamma_0$, with $\gamma_0$ small enough. 
Assume first that $N_\mu(\psi_\mu) = 1$. Using \eqref{eq:ineq2}, $\psi_\mu$ belongs to a set  $\Uc_\mu(\gamma)$ with a constant $\gamma$ depending on $\gamma_0$ and $\epsilon_0$. If these parameters are sufficiently small, we can define an element $u_\mu$ of $B_\mu(\rho_0)$ and $\alpha\in \T$ such that $\psi_\mu = \chi_\mu(\alpha,r_\mu(u_\mu), u_\mu)$ (recall that $N_\mu(\psi_\mu) = 1)$ with $u_\mu$ satisfying \eqref{eq:bdmerde2}. 
Hence we have 
$$
|ÊH_\mu(\psi_\mu) - H_\mu(\eta_\mu) | = | \Hc(u_\mu) - \Hc(u_\mu^*)|, 
$$ 
where $u_\mu^*$ is the minimizer of $\Hc$, associated with the discrete soliton $\eta_\mu$. This implies that 
there exists a constant $C$ independent of $\mu$ such that
$$
\Norm{u_\mu - u_\mu^*}{\mu}^2\leq C |ÊH_\mu(\psi_\mu) - H_\mu(\eta_\mu) | .
$$
Then using that $\Norm{u_\mu^*}{\mu} \leq C \epsilon(\mu)$,  that $\Norm{u_\mu}{\mu} = \dist_\mu(\psi_\mu,\pi_\mu \Gamma)+O(\epsilon(\mu))$, and the inequalities \eqref{eq:ineq1} and \eqref{eq:ineq2} we obtain \eqref{eq:ctrl} in the case $N_\mu(\psi_\mu) = 1$. Now if $N_\mu(\psi_\mu) \neq 1$ but $|N(\psi_\mu) - 1 | \leq \delta$ with $\delta$ sufficiently small, there exists a point $v_\mu$ such that $\Norm{v_\mu}{\mu} \leq \delta$ and $N(\psi_\mu -v_\mu) = 1$. We can then apply the previous estimate to $\psi_\mu - v_\mu$, and we use the uniform bounds on the derivative $H_\mu$ to conclude. The approximation $\psi_\mu\sim v_\mu$ gives rise to  the terms $C \delta$ in \eqref{eq:ctrl}. 
\end{proof}

\subsection{Discrete orbital stability}


In the previous paragraph, we have shown that the conditions \textbf{(i)--(iv)} are sufficient to ensure the existence of a modified soliton for the modified energy $H_\mu$, and that this soliton is sufficiently close to the exact soliton $\eta$ to allow the control of the distance between $\Gamma$ and $\psi_\mu$ via the distance between the Hamiltonian of $H_\mu(\psi_\mu)$ and $H_\mu(\eta_\mu)$, see \eqref{eq:ctrl}. As a consequence we obtain the following stability result

\begin{theorem}
\label{th:stab}
Assume that the hypothesis {\em{\textbf{(i)--(iv)}}} are satisfied, and assume moreover that for all $R_0$ and all $\mu \in \Sigma$ there exist $\beta(\mu) >0$ and an application $\Phi_ \mu: B_\mu(R_0) \to V_ \mu$  such that 
$$
\forall\, \psi_\mu \in B_\mu(R_0), \quad 
N_ \mu(\Phi_ \mu(\psi_ \mu)) = N_ \mu(\psi_ \mu)
$$
and
\begin{equation}
\label{eq:presham}
\forall\, \psi_\mu \in B_\mu(R_0), \quad |H_ \mu(\Phi_ \mu(\psi_ \mu)) - H_ \mu(\psi_ \mu)| \leq \beta( \mu). 
\end{equation}
Then there exist $\delta_0> 0$ and a constant $C$ such that for all positive $\delta < \delta_0$ and all $\mu\in \Sigma$ such that $\epsilon(\mu) < \epsilon_0$ and $\psi^0_\mu$ satisfying $\dist(i_ \mu \psi_ \mu^0,\Gamma) \leq  \delta$ then the sequence $(\psi_\mu^{n})_{n \geq 0}$ defined by  
$$
\forall\, n \geq 0, \quad 
\psi^{n+1}_ \mu = \Phi_ \mu(\psi^n_ \mu)
$$
satisfies
$$
\forall\, n \geq 0, \quad \dist(i_ \mu \psi_ \mu^n,\Gamma)  \leq C( \delta + \epsilon( \mu))
$$
as long as $n \beta(\mu) \leq \varepsilon( \mu) +  \delta$.
\end{theorem}

\begin{proof}
Using the hypothesis on $\psi_\mu^0$ and \eqref{eq:normmu}, there exists $R_0$ depending only on $\delta_0$ such that $\psi_\mu^0 \in B_\mu(R_0/2)$ uniformly in $\mu$ and there exists $\tilde \nu\in\Gamma$ such that $\Norm{i_\mu\psi_\mu^0 - \tilde \nu}{H^1} \leq \delta$. Thus using the gauge invariance of $H$, we have $|H(i_\mu \psi_\mu^0)-H(\eta)|\leq C\delta$. Then with hypothesis {{\textbf{(iii)}}} and \eqref{5.7bis}, we get
$$
|H_\mu(\psi_\mu^0) - H_\mu(\eta_\mu)| \leq C (\delta + \epsilon(\mu)). 
$$ 
On the other hand, using \eqref{eq:presham}, we have for all $n \geq 0$
\begin{eqnarray*}
|H_\mu(\psi_\mu^n) - H_\mu(\eta_\mu)| &\leq & |H_\mu(\psi_\mu^0) - H_\mu(\eta_\mu)| + \sum_{k = 0}^{n-1} |H_\mu(\psi_\mu^{k+1}) - H_\mu(\psi_{\mu}^{k})|\\
&\leq& C (\delta + \epsilon(\mu)) + n \beta(\mu) \leq (C+1)( \delta + \epsilon(\mu))
\end{eqnarray*}
as long as $n \beta(\mu) \leq \varepsilon( \mu) +  \delta$ and $\Norm{\psi_\mu}{\mu} \leq R_0$. Using the fact that $N_\mu(\psi_\mu^n) = N_\mu(\psi_\mu^0) = 1 + \mathcal{O}(\delta)$ and \eqref{eq:ctrl}, we get
\begin{equation}
\label{eq:poc}
 \dist(i_ \mu \psi_ \mu^n,\Gamma)   \leq \tilde{C}( \delta + \epsilon(\mu))
\end{equation}
as long as $\Norm{\psi_\mu}{\mu} \leq R_0$ and for some constant $\tilde C$ independent of $\mu$ and $n$. Then by a bootstrap argument,  there exists $\delta_0$ and $\epsilon_0$ sufficiently small such that, for $0<\delta<\delta_0$ and $0<\epsilon<\epsilon_0$, \eqref{eq:poc} ensures that this is the case for $n \beta(\sigma) \leq \varepsilon( \mu) +  \delta$. This proves the result. 
\end{proof}

\section{Applications}

We now prove the three Theorems presented in Section 2. We only need to verify the hypothesis \textbf{(i)-(iv)} and to precise the constants $\epsilon(\mu)$ and $\beta(\mu)$.

\subsection{Discrete Schr\"odinger equation}

Consider the DNLS equation \eqref{dnls} for a given positive number $h > 0$. In the previous formalism, we set $\Sigma = \{ h \in \R^+\}$, and the natural modified Hamiltonian and $L^2$ norm are given by \eqref{eq:discrham}. 
We also define the real scalar product 
$$
\langle \psi, \varphi\rangle_h := \Re \Big(h \sum_{j \in \Z} \psi_j \overline{\varphi_j} \Big). 
$$
For all $\mu \in \Sigma$, the embedding $i_h$ is defined by \eqref{eq:ih}, and the projection $\pi_h$ by the application
$$
\forall\, j \in \Z,\quad (\pi_h \psi)_j = \psi(jh),
$$
for some $\psi \in H^1$. 
Defining the semi norm 
$$
\SNorm{\psi}{h}^2 = 2h \sum_{j\in\Z} \frac{|\psi_{j+1}-\psi_j |^2}{h^2}
$$
on $V_h$, we have by explicit calculation that 
\begin{equation}
\label{eq:seminorm}
\SNorm{\psi}{h} = \SNorm{i_h \psi}{H^1} 
\end{equation}
where $\SNorm{\psi}{H^1}$ denotes the semi norm in $H^1$.  This fact allows to prove \textbf{(i)} and \textbf{(ii)} with the function  $\epsilon: \Sigma \to \R$ defined by $\epsilon(h) = h$. This has already been proved in \cite[Lemma 4.2]{BP10}. Similarly, \textbf{(iii)} has been proved in \cite[Proposition 4.1]{BP10} with $\epsilon(h) = h$. 

Finally, by classical arguments on finite elements approximation, there exists an universal constant $C$ such
that for any function $\psi\in H^2$
\begin{equation}
\label{estim:discret}
\Norm{\pi_h\psi -\psi}{H^1}\leq C h
\Norm{\psi}{H^2}.\end{equation}
 This proves \textbf{(iv)} upon using \eqref{eq:normmu}. 

Let us define $\Phi_h^t(\psi)$ the flow associated with the Hamiltonian $H_h$. Using standard estimates, one  shows that it is well defined for sufficiently small $t$, say $0\leq t<t_0$, uniformly in $h$.  Theorem \ref{th:dnls} is then a consequence of Theorem \ref{th:stab} with $\beta(h) = 0$ and $\Phi_\mu = \Phi_h^t$ with $t \in (0,t_0)$. Remark that, in particular, since $\Phi^{nt}_\mu=(\Phi^t_\mu)^n$ remains localized around the curve $\Gamma$ of solitons for all $n$ and for all $t\in(0,t_0)$, the flow $\Phi_h^t(\psi)$ is defined globally.

\subsection{Dirichlet cut-off}

Recall that in comparison with the previous case, the space $V_{h,K}$ defined in \eqref{eq:VhK} is a finite dimensional space included in $V_h$. We have seen that the modified energy and norm $H_{h,K}$ and $N_{h,K}$, and the embedding $i_{h,K}$ are defined by restriction to $V_{h,K}$. To define the projection $\pi_{h,K}$, we set 
$$
(\pi_{h,K}(\psi))_j = 
\left\{
\begin{array}{ll}
\psi(jh) & \mbox{if}\quad |j| \leq K\\[2ex]
0 & \mbox{if}\quad |j|  > K. 
\end{array}
\right.
$$
With these definitions, it is clear that the hypothesis \textbf{(i)-(iii)} are satisfied with $\Sigma = \{(h,K) \in \R^+, \times \N\}$ and with {\em a priori} $\epsilon(\mu) = h$ for $\mu = (h,K)$. However, the estimate \eqref{eq:approxeta} is no longer true with the space cut-off without changing the definition of $\epsilon(\mu)$. 

To have an estimate of $\Norm{i_{h,K} \pi_{h,K} \eta - \eta}{H^1}$, we only need to estimate $\Norm{\pi_{h,K}\eta - \pi_{h}\eta}{h}$ which is equal to 
$$
\Norm{\pi_{h,K}\eta - \pi_{h}\eta}{h}^2 = 2 h \sum_{|j| > K} \frac{|\eta(jh)|^2}{h^2} + h \sum_{|j| > K}|\eta(jh)|^2
$$
By definition of $\eta$, there exist  constants $C_1$ and $\nu$ such that for all $x \in \R$, $|\eta(x)| \leq C_1 e^{-\nu|x|}$. 
Substituting this estimate in the previous one, we get
\begin{eqnarray*}
\Norm{\pi_{h,K}\eta - \pi_{h}\eta}{h}^2 &\leq&  2 C_1^2 h \sum_{|j| >K} \frac{e^{-2 \nu jh}}{h^2} + h \sum_{|j| > K} e^{-2 \nu jh}\\
&\leq & \frac{4 C_1^2 + 2}{h^2} h \sum_{n > K} e^{-2 \nu n h}\\
&\leq&  \frac{\gamma}{h^2} \exp(- \nu K h) 
\end{eqnarray*}
for some constant $\gamma$, and provided $h < h_0$ sufficiently small. 

This shows that \textbf{(iv)} is valid with the function 
\begin{equation}
\label{eq:epsi1}
\epsilon(\mu) = h + \frac{1}{h^2} \exp(-\nu Kh), \quad \mu = (h,K) \in \Sigma
\end{equation}
With these notations, Theorem \ref{th:dnlsdir} is a consequence of Theorem \ref{th:stab} with $\beta(\mu) = 0$.

%

\subsection{Time splitting method}

Let us now consider the case where \eqref{dnlsdir} is discretized in time by a splitting method of the form $\Phi_A^\tau \circ \Phi_P^\tau$ as described in Section 2. The space discretization being the same as in the previous Sections, the hypothesis \textbf{(i)-(iii)} will be automatically fulfilled with the function $\epsilon$ defined in \eqref{eq:epsi1}. In particular, we can check directly that the norm $N_{h,K}$ is preserved by splitting schemes. However, splitting methods do not preserve the energy $H_{h,K}$ for given $h$ and $K$: more precisely, taking $H_\mu = H_{h,K}$ in \eqref{eq:presham} only yields an error of order $\beta(\mu) = h\tau$. 

In this section, we set 
$$
\Sigma := \{ (h,K,\tau) \in \R^+ \times \N \times \R^+\}.
$$
For $\mu = (h,K,\tau) \in \Sigma$, we set $V_{\mu} = V_{h,K}$,  $i_{\mu} = i_{h,K} = i_h$, and $\pi_{\mu} = \pi_{h,K}$. 

In the next section we will prove
\begin{theorem}
\label{prop.split}
Let  $R_0>0$ and $M \in \N$ be fixed. There exist $\tau_0$ and $h_0$ such that for all $\mu = (h,K,\tau) \in \Sigma$ satisfying $\tau < \tau_0$, $h < h_0$ and 
\begin{equation}
\label{sp.0cfl}
(2M+3)\frac{\tau}{h^2}<\frac{2\pi}{3}. 
\end{equation}
then there exist a constant $C$, depending only on $R_0$ and  $M$, and a smooth gauge invariant polynomial function
$H_{\mu} = H_{h,K,\tau}$ defined on $V_\mu$ 
 such that
\begin{equation}
\label{sp.01}
\Norm{H_{\mu} -H \circ i_{\mu}}{C^2(B_{\mu}(R_0))}\leq
 C\frac{\tau}{h} 
 \end{equation} 
and
\begin{equation}
\label{sp.02}
\Norm{\Phi_P^\tau\circ\Phi_A^\tau(\psi)-\Phi^\tau_{H_\mu}(\psi)}{\mu}\leq C\tau^{M+1}\quad \text{for all }\psi\in V_\mu \text{ with } \Norm{\psi}{\mu}\leq R_0.
\end{equation} 
\end{theorem}
With this result, the final statement of Theorem \ref{th:splitting} is a consequence of Theorem \ref{th:stab} applied with 
$$
\epsilon(\mu) = h + \frac{1}{h^2} \exp(-\nu Kh) + \frac{\tau}{h}
$$
and $\beta(\mu) = \tau^{M+1}$. 
The proof of Theorem \ref{prop.split} occupies the rest of this paper, and is a variant of the
theory developed in \cite{BG94,FG11}. Here we summarize it and repeat the
proofs with some details in order to have a quite self contained
presentation.

\section{Construction of the modified energy}

\subsection{Formal part. } We start by recalling the algorithm of
construction of the modified energy $H_\mu$ introduced in the previous section. As a variant of the theory developed in \cite{FG11}, we work here at the level of the vector fields instead of Hamiltonian functions. Recall that at the continuous level, we identified the space $H^1(\R;\C) \simeq H^1(\R;\R)^2$ through the identification $\psi = \frac{1}{\sqrt{2}}(q + ip)$. This identification obviously transfers to the space $V$ of symmetric functions, and to the discretized space $V_{h,K}$ via the identification 
\begin{equation}
\label{eq:idef}
\psi_j = \frac{1}{\sqrt{2}}(q_j + i p_j), \quad j = -K,\ldots,K.
\end{equation} 
Hence we can endow $V_{h,K}$ with the Hamiltonian structure induced by the symplectic form 
$\sum_{j = -K}^K \dd p_j \wedge \dd q_j$. In the following we make the constant identification between $\psi = (\psi_j)_{j = -K}^{K}$ and $(q,p) = (q_j,p_j)_{j = -K}^K$ given by \eqref{eq:idef}. For a given real functional $H(\psi) = H(q,p)$, we associate the Hamiltonian vector fied $X_H$ by
\begin{equation}
\label{eq:XH}
X_H(q,p):=\left( \frac{\partial H}{\partial p_\l}(q,p),-\frac{\partial H}{\partial q_\l}(q,p)\right)_{\l=-K}^K.
\end{equation}
Note that this formula makes sense, because all the Hamiltonian functions $H(\psi)$ that we consider are real valued. 

In this setting, $A$ and $P$ denote the vector fields associated respectively to the real Hamiltonian functions 
$$
H_A(\psi)= h\sum_{\l=-K}^K \frac{|\psi_\l-\psi_{\l-1}|^2}{h^2},\quad \mbox{and}\quad
H_P(\psi)= -\frac{h}{2}\sum_{\l=-K}^K |\psi_\l|^4, 
$$
which can obviously be expressed in terms of $(q_j,p_j)$. Note that $A$ and $P$ depend on $h$, but we omit this dependence in the notation. 
We look for
a formal vector field, namely a formal power series
\begin{equation}\label{sp.0}
Z(\eps):=\sum_{n\geq 0}Z_j\eps^n, 
\end{equation}
where each $Z_n$ is a Hamiltonian vector field on $V_{h,K}$,
such that 
\begin{equation}
\label{sp.1}
\forall \, |\eps|\leq \tau, \quad
\Phi^\eps_{P}\circ\Phi^1_{A_0}=\Phi^1_{Z(\eps)}, \quad A_0:=\tau A\ .
\end{equation}
Here $\Phi_X^t$ denotes the Hamiltonian flow on $V_{h,K}$ associated with the vector field $X$ at time $t$. \\
Notice that, in particular, at order zero \eqref{sp.1}  implies 
\begin{equation}
\label{sp.1.1}
Z_0:=A_0=\tau A\ .
\end{equation}
Ideally, the approximate Hamiltonian we are looking for would be $H_{h,K,\tau}:=\frac 1 {\tau} H_{Z(\tau)}$ (see \eqref{Hamtau}) but the formal  series defining $Z$ is not convergent and we will have to truncate the sum in  \eqref{sp.0}.\\
It is well known that it is convenient to look at the
equality \eqref{sp.1} in a dual way, namely to ask that the following
equality is fulfilled for any smooth function $w:V_{h,K}\to \C$:
\begin{equation}
\label{sp.2}
w(\Phi^\eps_{P}\circ\Phi^1_{A_0}  )=w(\Phi^1_{Z(\eps)}
)\ .
\end{equation} 
The key ingredient of the construction is given by the formal
formula
\begin{align}
\label{sp.3}
\forall\, t, \quad 
e^{t L_X}w=w\circ \Phi^t_X\ ,
\end{align}
where $L_X$ is the Lie operator associated with $X$. In our Hamiltonian case if  $X:=(X^{j}_q,X^{j}_p)_{j=-K}^K$ is a vector field (according to the decomposition \eqref{eq:XH}), we have in real coordinated $(q_j,p_j)$, 
$$
L_X w:=\sum_{j = -K}^K X^{j}_p\frac{\partial w}{\partial p_j}-X^{j}_q\frac{\partial w}{\partial q_j}, 
$$
and the exponential is defined in a formal way by
$$
e^{\eps L_X}w:=\sum_{k\geq 0}\frac{1}{k!}\eps^kL^k_Xw\ .
$$
In this formalism \eqref{sp.2} reads
$$
e^{L_{A_0}}e^{\eps L_P}w=e^{L_{Z(\eps)}}w.
$$
Deriving with respect to $\eps$ one gets (by working on the power
series)
\begin{equation}
\label{sp.5}
e^{L_{A_0}}e^{\eps L_P} L_Pw=e^{L_{Z(\eps)}}L_{Q(\eps)}w,
\end{equation}
where 
\begin{equation}
\label{sp.6}
Q(\eps):=\sum_{k\geq
0} \frac{1}{(k+1)!}\ad^k_{Z(\eps)}Z'(\eps)\quad \mbox{with}\quad \ad _Z X:=\left[ Z, X\right], 
\end{equation}
where $[\, \cdot , \cdot \,]$ denote the Lie bracket of two vector fields. 
 Finally \eqref{sp.5} leads to
the equation $Q(\eps)=P$ from which we are going to construct
$Z(\varepsilon)$. The construction goes as follows: first one remarks that the
r.h.s. of \eqref{sp.6} has the formal aspect of an operator applied to
$Z'(\varepsilon)$, so the idea is first of all to invert such an operator. We
remark that the power series defining the wanted operator is
$\sum_{k\geq 0}x^k/(k+1)!=(e^x-1)/x$, so that one would expect its
inverse to be $x/(e^x-1)\equiv \sum_{k\geq0} x^k(B_k/k!)$, where $B_k$
are the so called Bernoulli numbers and the power series is convergent
provided $|x|<2\pi$. So one is tempted to rewrite $Q(\eps)=P$ in the
form
\begin{equation}\label{Ehomo2}
\forall\, |\varepsilon| \leq \tau, \quad Z'(\varepsilon) = \sum_{k \geq 0} \frac{B_k}{k!} {\ad}^k_{Z(\varepsilon)} P. 
\end{equation}

Plugging an Ansatz expansion $Z(\varepsilon) = \sum_{\ell \geq 0} \eps^\ell Z_\ell$
into this equation, we get, for $n \geq 0$, the recursive equations
\begin{equation}
\label{Erec}
(n+1) Z_{n+1} =  \sum_{k \geq 0 }\frac{B_k}{k!} A_k^{(n)}, \quad 
\mbox{with}\quad 
A_k^{(n)}:=\sum_{\ell_1 + \cdots + \ell_k =
n} \mathrm{ad}_{Z_{\ell_1}} \cdots \mathrm{ad}_{Z_{\ell_k}} P .
\end{equation}

\begin{remark}
The analysis made to obtain this recursive equation is formal. To
obtain our main result, we will verify that some of the series we manipulate
are in fact convergent series, while the others will be truncated in
order to get meaningfull expressions.
\end{remark}

\begin{remark}
\label{order}
Assume that $P$ is a polynomial of degree $r_0$ (in our case $r_0=3$), and that $Z_\ell$ is a collection of vector fields satisfying the previous relation, 
then for all $n$, $Z_n$ is a polynomial of degree $(n-1)(r_0-1)+r_0$. 
\end{remark}
\begin{remark}
\label{ham}
If the vector fields $P$ and $A_0$ are Hamiltonian then the same is
true for the vector fields $Z_n$. This is an immediate consequence of
the fact that all the construction involves only Lie Brackets, which
are operations preserving the Hamiltonian nature of the vector
fields. 
\end{remark}

\subsection{Analytic estimates}\label{analytic} We first introduce a
suitable norm for measuring the size of the polynomials. In echo
with the notations of the previous sections, we consider in the
following a fixed $\mu = (h,K,\tau) \in \Sigma$. Recall that the space
$V_\mu = V_{h,K}$ does not depend on $\tau$, as well as the norm
$\Norm{\cdot }{\mu}$.  
If $X$ is a vector field on $V_{\mu}$ which is a homogeneous
polynomial of degree $s$ we
can associate to it a symmetric multilinear form
$\widetilde{X}(\psi_1,\ldots,\psi_{s_1})$ such that $X(\psi)
= \widetilde{X}(\psi,\ldots,\psi)$. We put
$$
\Norm{X}{\mu}:=\sup_{\substack{\Norm{\psi_i}{\mu} = 1 \\Êi =
1,\ldots,s_1}} \Norm{\widetilde{X}(\psi_1,\ldots,\psi_{s_1})}{\mu} \ .
$$ 
We then extend this norm to general polynomial vector field $X$ by
defining its norm as the sum of the norms of the homogeneous
components.

\begin{definition} We denote by $\P_s$ the space of the
 polynomials of degree less than $s$, which furthermore have a finite norm
$\Norm{\cdot}{\mu}$. \end{definition}

\begin{remark}
\label{rk.important}
With this definition, we note that the norm $\Norm{P}{\mu}$ is
uniformly bounded with respect to $\mu$.
\end{remark}
\begin{lemma}
\label{ad.1}
Let $s_1 \geq 1$ and $s_2\geq 1$, and let $X\in\P_{s_1}$ and
$Y\in\P_{s_2}$. Then $[X,Y] \in \P_{s_1 + s_2 - 1}$, and
\begin{equation}
\label{sp.lem.1.1}
\Norm{[X,Y]}{\mu} \leq (s_1 + s_2) \Norm{X}{\mu} \Norm{Y}{\mu}. 
\end{equation}
\end{lemma}
\proof We give the proof in the case of homogeneous polynomials, the
general case immediately follows. Denote again by $\widetilde X$ and
$\widetilde Y$ the symmetric multilinear forms associated to $X$ and
$Y$, then one has
$$
[X,Y](\psi)=s_1\widetilde{X}(Y(\psi),\psi...,\psi)-s_2\widetilde{Y}(X(\psi),\psi...,\psi), 
$$
from which the result immediately follows. \qed

\begin{lemma}
\label{ad.2}
For $h \leq \frac{1}{\sqrt{2}}$, the operator $A_0 = -\tau \Delta_h$ satisfies
\begin{equation}
\label{sp.lem.df}
\Norm{A_0}{\mu} 
\leq  3 \frac{\tau}{h^2}. 
\end{equation}
\end{lemma}
\proof  
Let us first note that if $(u_j)_{j = -K}^K$ is in $V_\mu$, we have 
\begin{equation}
\label{eq:discrSob}
\Norm{u}{\mu}^2 = 2 h \sum_{j  = -K}^K \frac{|u_{j+1} - u_j|^2}{h^2} + h\sum_{j = -K}^{K} | u_j|^2 \leq (\frac{4}{h^2} + 1) \Big(h\sum_{j = -K}^{K} | u_j|^2\Big).
\end{equation}
Note that $A_0=-\tau\Delta_\mu$ is homogeneous of degree one. Moreover, we can write 
$$
(A_0\psi)_\l= \tau\frac{\psi_{\l+1}+\psi_{l-1}-2\psi_\l}{\mu^2} = \frac{\tau}{h} (a_{\ell} - a_{\ell-1}),
$$
where $a_{\ell} = (\psi_{\ell +1} - \psi_{\ell})/h$. Using the discrete Sobolev inequality \eqref{eq:discrSob} and the Minkowski inequality, we get that 
$$
\Norm{A_0 \psi}{\mu} \leq 2 \sqrt{(\frac{4}{h^2} + 1)} \frac{\tau}{h}\Big(h\sum_{j = -K}^{K} | a_j|^2\Big)^{1/2}. 
$$
We conclude by remarking that 
$$
\Big(h\sum_{j = -K}^{K} | a_j|^2\Big) \leq \frac{1}{2} \Norm{\psi}{\mu}^2. 
$$
We deduce that 
$$
\Norm{A_0\psi}{\mu} \leq \frac{\tau}{h^2}\sqrt{8 + 2
h^2}\norma{\psi_\mu}\ ,
$$
which shows the result.
\qed
\begin{remark}
Lemmas \ref{ad.1} and \ref{ad.2} can be rephrased in a form suitable
for the following by saying that, for $X\in\P_s$, one has that the
operator
\begin{equation*}
\ad_X:\P_{s_1}\to\P_{s+s_1-1}
\end{equation*}
is bounded and its norm (induced by the norm $\Norm{\cdot}{\mu}$ and for fixed $s$ and $s_1$) fulfills
\begin{equation}
\label{ad.3}
\Norm{\ad X}{\mu}\leq (s+s_1)\Norm{X}{\mu}\ .
\end{equation}
In particular, using the previous result we have for a given $s_1 \geq 2$ 
\begin{equation}
\label{ad.4a}
\ad_{A_0}:\P_{s_1}\to\P_{s_1}
\quad \mbox{and}\quad 
\norma{\ad_{A_0}}_\mu\leq 3(s_1+1)\frac{\tau}{\mu^2}\ .
\end{equation}
\end{remark}
\begin{proposition}
\label{sp.imp}
Let $M$ be an integer  satisfying  
\begin{equation}
\label{boh}
(2M+3)\frac{\tau}{\mu^2}< \frac{2\pi}{3}\ .
\end{equation} 
Then, for all $ n\leq M$, $Z_n$ is well defined and
$Z_n\in\P_{r_n}$ with $r_n=2n+1$, and the norm of $Z_n$ is uniformly bounded with respect to $\mu$. 
\end{proposition}
\proof We prove the proposition by induction. We set $Z_0 = A_0$. Assume that $Z_\l\in\P_{r_\l}$ for $\l\leq n \leq M-1$ are constructed. Let us prove that \eqref{Erec} defines a term $Z_{n+1} \in \P_{r_{n+1}}$. 
Rewrite \eqref{Erec} 
by incorporating the terms containing $Z_0=A_0$ and by substituting the
estimate of the single terms to the $\ad$ terms. The advantage of
doing that is that the product of the estimates is commutative, while
the multiplication of the $\ad$ operators is not.
We get first
\begin{align*}
\Norm{A^{(n)}_k}{\mu}\leq \sum_{i=1}^{k}\Norm{\ad_{A_0}}{\mu}^{k-i} 
\frac{k!}{(k-i)!i!}\sum_{\substack{\ell_1+...+\ell_i=n \\ \ell_j\geq
1}}\Norm{\ad_{Z_{\ell_1}}}{\mu}...\Norm{\ad_{Z_{\ell_i}}}{\mu} \Norm{P}{\mu}
\\
\leq    \sum_{i=1}^{n}\Norm{\ad_{A_0}}{\mu}^{k-i} 
\frac{k!}{(k-i)!i!}\sum_{\substack{\ell_1+...+\ell_i=n \\ \ell_j\geq
1}} (2r_M)^n\Norm{{Z_{\ell_1}}}{\mu}...\Norm{{Z_{\ell_i}}}{\mu} \Norm{P}{\mu},
\end{align*}
where we used that, if $i>n$ and $\ell_j>0$ then $\ell_1+...+\ell_i>n$ and the
fact that, since by hypothesis the involved polynomials have degrees
smaller then $r_M$, one has $\Norm{\ad_{Z_\ell}}{\mu}\leq
2r_M\norma{Z_\ell}{\mu}$ for $\ell \leq n$. 

Remarking that the result of the above sum with respect to
$\l_1,\cdots,\l_i$ does not depend on $k$, using  \eqref{ad.4a} with $s_1=r_M$, and noticing that $\Norm{P}{\mu}$ is uniformly bounded with respect to $\mu$, we get
\begin{align*}
\Norm{Z_{n+1}}{\mu}&\leq  \frac{1}{n+1}\sum_{k\geq
0}\frac{B_k}{k!} \sum_{i=0}^{n} \left(\frac{r_M\tau}{\mu^2}\right)^{k-i}
 \frac{k!}{(k-i)!i!} C_n\\
 &= \frac{C_n}{n+1}\left[ \sum_{i=0}^{n}\frac{\dd^i}{\dd x^i} \left(\sum_{k\geq
0}\frac{B_k}{k!}x^k\right)  \right]_{x=\frac{r_M\tau}{\mu^2}},
\end{align*}
for some constant $C_n$ independent of $\mu$. 
This shows that the series defining $Z_{n+1}$ is convergent, that $Z_{n+1}\in\P_{r_{n+1}}$ and that $\Norm{Z_{n+1}}{\mu}$ is finite and uniformly bounded with respect to $\mu$.  \qed

\subsection{Proof of Theorem \ref{prop.split} }
First remark that in our case all the vector fields are
Hamiltonian. Explicitely, by Poincar\'e Lemma, the Hamiltonian
function of a Hamiltonian vector field $X$ is given by
\begin{equation}
\label{ham.x}
H_X(\psi):=\int_0^1s\omega(X(s\psi),\psi) \dd s\ ,
\end{equation} 
where $\omega$ is the symplectic form. In particular, this formula 
shows that the Hamiltonian function of a smooth polynomial vector
field is also a smooth polynomial function. For $\varepsilon \leq \tau$, let us define $$Z^{(M)}(\eps):=\sum_{j=0}^{M}\eps^jZ_j\ .$$
By construction $Z^{(M)}(\eps)$ satisfies \eqref{Ehomo2} up to order $\eps^{M}$ included from which we deduce that it satisfies  \eqref{sp.1} up to order $\eps^{M}$ (see \cite{FG11} Theorem 4.2 for details).
Therefore defining for $\mu = (h,K,\tau)$, 
\begin{equation}
\label{Hamtau}
H_\mu:=\frac 1 \tau H_{Z^{(M)}(\tau)}=\sum_{j=0}^{M}\tau^{j-1}H_{Z_j}\ ,
\end{equation}
estimate \eqref{sp.02} holds true with a constant independent of $\mu$. 

It remains to compare the two Hamiltonians $H_\mu = H_{h,K,\tau}$ and $H_{h,K}$ in the $C^2$ norm on the ball centered at the origin and of arbitrary radius $R_0$ in $V_\mu$.\\
Let us define
$$
H^{(1)}_\mu =\frac 1 \tau ( H_{Z_0}+\tau H_{Z_1})$$
and recall that $Z_0=A_0=\tau A$, and that by construction
\begin{equation}
\label{eq:Z1}
Z_1= \sum_{k\geq 0}\frac{\tau^k B_k}{k!} \ad_{A_0}^k P\ .
\end{equation}
Now we have 
$$
H_\mu^{(1)}-H_\mu= \sum_{j=2}^{M}\tau^{j-1}H_{Z_j}. 
$$
But using \eqref{ham.x} and the fact that $Z_{j}$ is of degree $r_j$, we get for all $\psi \in B_{\mu}(R_0)$, 
\begin{equation}
\label{h1}
|H_\mu^{(1)}(\psi)-H_\mu(\psi)|\leq \sum_{j=2}^{M}\tau^{j-1}\Norm{Z_j}{\mu}  R_0^{r_{j} +1 }\leq C\tau\ , 
\end{equation}
for some constant $C$ independent of $h$, $K$ and $\tau \leq \tau_0$ sufficiently small. 
To estimate $H_\mu^{(1)} - H_{h,K}$, we notice using \eqref{eq:Z1}, 
\begin{equation}
X_{H^{(1)}_\mu}-X_{H_{h,K}}= Z_1-P
=\tau \left[\sum_{k\geq 0}\frac{\tau^k B_{k+1}}{(k+1)!} \ad_{A_0}^k\right]\ad_{A_0} P.
\end{equation}
But in view of \eqref{sp.0cfl}, $\frac{3\tau}{\mu^2}< \pi$, and thus the operator $ \left[\sum_{k\geq 0}\frac{\tau^k B_{k+1}}{(k+1)!} \ad_{A_0}^k\right]$ is bounded on $\P_3$, uniformly with respect to $\mu$.  Therefore for $\psi \in B_\mu(R_0)$, we have  
$$
|H^{(1)}_\mu(\psi)-H_{h,K}(\psi)|\leq C | H_{\ad_{A_0}P}(\psi)|.$$
for some constant $C$ independent on $\mu = (h,K,\tau)$. 
Now we calculate explicitly that the Hamiltonian associated with $\ad_{A_0}P$ is given by 
\begin{align*}
H_{\ad_{A_0}P}(\psi)=&\frac{ i\tau}{ \mu^2} \sum_{-K}^K (\ov{\psi_{\l+1}}+\ov{\psi_{\l-1}}-2\ov{\psi_\l})|\psi_\l|^2\psi_\l-(\psi_{\l+1}+\psi_{\l-1}-2\psi_\l)|\psi_\l|^2\ov{\psi_\l}\\
=&\frac{\tau}{ \mu^2} \sum_{-K}^K  \Im((\psi_{\l+1}+\psi_{\l-1}-2\psi_\l)|\psi_\l|^2\ov{\psi_\l}).
\end{align*}

But we have 
\begin{align*}
&\sum_{\ell  = -K}^K  (\psi_{\l+1}+\psi_{\l-1}-2\psi_\l)|\psi_\l|^2\ov{\psi_\l} \\
&= \sum_{\ell = -K}^K  (\psi_{\l+1} - \psi_\l)|\psi_\l|^2\ov{\psi_\l} - (\psi_{\l} - \psi_{\l-1})|\psi_\l|^2\ov{\psi_\l} \\
& = \sum_{\ell = -K}^K  (\psi_{\l+1} - \psi_\l)|\psi_\l|^2\ov{\psi_\l} -  \sum_{\ell = -K - 1}^{K-1} (\psi_{\l+1} - \psi_{\l})|\psi_{\l+1}|^2\ov{\psi_{\l+1}}\\
&= \sum_{\ell = -K}^{K-1}  (\psi_{\l+1} - \psi_\l)(|\psi_\l|^2\ov{\psi_\l} - |\psi_{\l+1}|^2\ov{\psi_{\l+1}})  - \psi_K|\psi_K|^2\ov{\psi_K} + \psi_{-K}|\psi_{-K}|^2\ov{\psi_{-K}}
\end{align*}
using the boundary conditions $\psi_{K+1} =  \psi_{-K - 1} = 0$. Taking the imaginary part, we obtain
$$
H_{\ad_{A_0}P}(\psi)
=\frac{\tau}{ \mu^2}\sum_{\ell = -K}^{K-1} \Im( (\psi_{\l+1} - \psi_\l)(|\psi_\l|^2\ov{\psi_\l} - |\psi_{\l+1}|^2\ov{\psi_{\l+1}}) ). 
$$
But we have
$$ 
\left| \Im( (\psi_{\l+1} - \psi_\l)(|\psi_\l|^2\ov{\psi_\l} - |\psi_{\l+1}|^2\ov{\psi_{\l+1}}) ) \right|Ê\leq
5 |Ê\psi_{\l+1} - \psi_\l |^2 (|\psi_\ell|^2 + |\psi_{\l + 1}|^2). 
$$
Then we use that 
$$|\psi_{\l+1}- \psi_\l|^2\leq \mu \Norm{\psi}{\mu}^2,$$
to obtain
$$|H_{\ad_{A_0}P}(\psi)|\leq 2\frac{\tau}{\mu}\Norm{\psi}{\mu}^4$$
and therefore for $\psiÊ\in B_\mu(R_0)$, 
\begin{equation}
\label{h2}
|H^{(1)}_\mu(\psi)-H_{h,K}(\psi)|\leq C\frac{\tau}{\mu}\ .
\end{equation}
Combining \eqref{h1} and \eqref{h2} we get, for all $\psiÊ\in B_\mu(R_0)$,
\begin{equation}
\label{h3}
|H_\mu(\psi)-H_{h,K}(\psi)|\leq  C\frac{\tau}{\mu}\ .
\end{equation}
Furthermore, since both functionals are analytic in $\psi$ and the above estimate is uniform in $\psiÊ\in B_\mu(R_0)$, we have similar estimates for the first and the second  derivative of $\psi\mapsto H_\mu(\psi)-H_{h,K}(\psi)$.
\qed


\begin{thebibliography}{99}

\bibitem{Akrivis}
{\rm G. D. Akrivis, V. A. Dougalis and O. A. Karakashian}, 
{\em On fully discrete Galerkin methods of second-order temporal accuracy for the nonlinear Schr\"odinger equation},
Numer. Math. 59 (1991) 31-53. 

\bibitem{BP10}
D.~Bambusi and T. Penati \emph{Continuous approximation of breathers in one and two dimensional DNLS lattices }, Nonlinearity 23 (2010), no. 1, 143Ð157.   

\bibitem{BG94}
{\rm G. Benettin and A. Giorgilli},
{\em On the Hamiltonian interpolation of near to the identity symplectic mappings with application to symplectic integration algorithms}, J. Statist. Phys. 74 (1994), 1117--1143. 

\bibitem{Besse}
{\rm C. Besse}, 
A relaxation scheme for the nonlinear Schr\"odinger equation,
SIAM J. Numer. Anal. 42 (2004) 934--952. 

\bibitem{Borgna08}
{\rm J. P. Borgna and D. F. Rial}
{\em Orbital stability of numerical periodic nonlinear Schr\"odinger equation}, 
Commun. Math. Sci. 6 (2008) 149--169. 

\bibitem{Ciarlet}
{\rm P.G Ciarlet, B. Miara and J.-M. Thomas}, 
{\em Introduction to numerical linear algebra and optimisation}, 
Cambridge University Press, 1989. 


\bibitem{DFP81}
{M. Delfour, M. Fortin, G. Payre}, 
Finite-difference solutions of a non-linear Schr\"odinger equation, 
J. Comput. Phys. 44 (1981) 277-288. 

\bibitem{Duran00}
{\rm A. Dur\'an and J. M. Sanz-Serna}, 
{\em The numerical integration of relative equilibrium solutions. The nonlinear Schr\"odinger equation}, 
IMA J. Numer. Anal. 20 (2000) 235-261. 

\bibitem{Fei95}
{\rm Z. Fei, V.M. P\'erez-Garc\'ia and L. V\'asquez},
{\em Numerical simulation of nonlinear Schr\"odinger systems: A new conservative scheme},
Appl. Math. Comput. 71 (1995) 165-177. 

\bibitem{F11}
{\rm E. Faou},
{\em Geometric numerical integration and Schr\"odinger equations}. European Math. Soc., 2012. 

\bibitem{FG11}
{\rm E. Faou and B. Gr\'ebert},
{\em Hamiltonian interpolation of splitting approximations for nonlinear PDE's}. 
Found. Comput. Math. 11 (2011) 381--415

\bibitem{Frohlich04}
{\rm J. Fr\"ohlich, S. Gustafson, L. Jonsson and I.M. Sigal}
{\em Solitary wave dynamics in an external potential}, Comm. Math. Phys. 250 (2004), 613--642

\bibitem{Grill87}
{\rm M. Grillakis, H. Shatah and W. Strauss}, 
{\em Stability theory of solitary waves in the presence of symmetry. I.},
J. Funct. Anal., 74 (1987) 160--197. 

\bibitem{Grill90}
{\rm M. Grillakis, H. Shatah and W. Strauss}, 
{\em Stability theory of solitary waves in the presence of symmetry. II.},
J. Funct. Anal., 94 (1990) 308--348. 

\bibitem{HLW}
{\rm E. Hairer, C. Lubich and G. Wanner},
{\em Geometric Numerical Integration. Structure-Preserving Algorithms for Ordinary Differential Equations}. Second Edition. Springer 2006. 


\bibitem{Reich99}
{\rm S. Reich},
{\em Backward error analysis for numerical integrators}, SIAM J. Numer. Anal. 36 (1999) 1549--1570. 

\bibitem{SZ84}
{\rm J. M. Sanz-Serna},
{\em Methods for the solution of the nonlinear Schroedinger equation},
Math. Comp. 43 (1984) 21--27. 

\bibitem{SZ86}
{\rm J. M. Sanz-Serna and J. G. Verwer},
Conservative and nonconservative schemes for the solution of the nonlinear Schr\"odinger equation,
IMA J. Numer. Anal. 6 (1986) 25-42. 


\bibitem{Weideman86}
{\rm J. A. C. Weideman, B. M. Herbst}, 
{\em Split-step methods for the solution of the nonlinear Schr\"odinger equation}, 
SIAM J. Numer. Anal. 23 (1986) 485-507.  

\bibitem{Weinstein85}
{\rm M. I. Weinstein}, 
{\em Modulational stability of ground states of nonlinear Schr\"odinger equations}, 
SIAM J. Math. Anal. 16 (1985) 472--491. 


\end{thebibliography}
\end{document}